\documentclass[11pt,twoside]{article}

\def\LL{\cc{L}}

\def\pen{\operatorname{pen}}

\def\pnn{\mathfrak{b}}

\def\penInt{\mathfrak{P}}
\def\penH{\mathfrak{H}}

\def\fu{f}
\def\tv{\bb{\upsilon}}
\def\Tv{\varUpsilon}
\def\tvs{\tv_{0}}
\def\td{\tv^{\circ}}
\def\tc{\tv^{\prime}}
\def\tdc{\tv^{\sharp}}

\def\normc{\mathfrak{d}}

\def\UU{\cc{Y}}
\def\ExpL{\mathfrak{D}}
\def\ExpM{\cc{M}}
\def\VV{H}

\def\ranking{\cc{R}}

\def\penb{\varkappa}

\def\Yn{\bb{Y}}
\def\eqdef{\stackrel{\operatorname{def}}{=}}
\def\mes{\pi}

\def\srho{s}
\def\lambdab{{\lambda}^{*}}

\def\fis{\mathfrak{a}}
\def\fiss{\fis_{1}}

\def\uv{\bb{u}}

\def\mY{b}

\def\Lexpm{\mathfrak{S}}
\def\expzeta{\mathfrak{n}}

\def\reps{\epsilon}
\def\repsc{\reps_{0}}

\def\lmgf{\mathfrak{m}}

\def\B{\cc{B}}
\def\BB{\B'}

\def\Ns{\mathbb{N}}

\def\rr{\mathfrak{r}}

\def\mrho{\varrho}
\def\Lmgf{\mathfrak{M}}

\def\ex{\mathrm{e}}
\def\E{I\!\!E}
\def\P{I\!\!P}

\def\entrl{\mathbb{Q}}

\def\entr{\entrl}

\def\thetav{\bb{\theta}}
\def\thetavs{\bb{\theta}_{0}}
\def\thetavc{\thetav'}
\def\thetavd{\thetav^{\circ}}

\def\gammav{\bb{\gamma}}

\def\thetas{\theta_{0}}

\def\kappa{\varkappa}
\def\zz{\mathfrak{z}}

\def\Crlp{\mathfrak{R}}
\def\Crlq{\mathfrak{Q}}

\newcommand{\ifims}[2]{#1}   

\def\thetitle{A penalized exponential risk bound in parametric estimation}
\def\theruntitle{A penalized exponential risk bound in parametric estimation}

\def\theabstract{
The paper offers a novel unified approach to studying the accuracy of parameter
estimation by the quasi likelihood method.
Important features of the approach are:
(1) The underlying model {is not assumed to be parametric}.
(2) No conditions on parameter identifiability are required.
The parameter set can be unbounded.
(3) The model assumptions are quite general and there is no specific structural
assumptions like independence or weak dependence of observations.
The imposed conditions on the model are very mild and can be easily checked in specific
applications.
(4) The established risk bounds are {nonasymptotic} and valid for large,
moderate and small samples.
(5) The main result is the  concentration property of the quasi MLE
giving an nonasymptotic exponential bound for the probability that the
considered estimate deviates out of a small neighborhood of the ``true'' point.

In standard situations under mild regularity conditions,
the usual consistency and rate results can be easily obtained as
corollaries from the established risk bounds.
The approach and the results are illustrated on the example of generalized linear
and single-index models.
}

\def\kwdp{62F10}
\def\kwds{62J12,62F25,62H12}

\def\thekeywords{exponential risk bounds, rate function,
quasi maximum likelihood}

\usepackage{amsmath,amssymb,amsthm}
\usepackage{natbib}
\usepackage{epsfig,graphicx}
\usepackage{comment}
\usepackage{color}
\usepackage{srcltx}
\renewcommand{\Gamma}{\varGamma}
\renewcommand{\Pi}{\varPi}
\renewcommand{\Sigma}{\varSigma}
\renewcommand{\Delta}{\varDelta}
\renewcommand{\Lambda}{\varLambda}
\renewcommand{\Psi}{\varPsi}
\renewcommand{\Phi}{\varPhi}
\renewcommand{\Theta}{\varTheta}
\renewcommand{\Omega}{\varOmega}
\renewcommand{\Xi}{\varXi}
\renewcommand{\Upsilon}{\varUpsilon}
\renewcommand{\(}{$\,}
\renewcommand{\)}{\,$}
\usepackage[mathscr]{eucal}

\newcommand{\cc}[1]{\mathscr{#1}}
\newcommand{\bb}[1]{\boldsymbol{#1}}

\renewcommand{\tilde}[1]{\widetilde{#1}}
\newcommand{\nn}{\nonumber \\}

\newcommand\Var{\operatorname{Var}}

\newcommand{\argmax}{\operatornamewithlimits{argmax}}
\def\argmin{\operatornamewithlimits{argmin}}

\def\E{\boldsymbol{E}}
\def\P{\boldsymbol{P}}

\def\R{I\!\!R}

\def\T{\top}

\newcommand{\keywords}[1]{\par\noindent\emph{Keywords:} #1 \\}

\renewenvironment{abstract}
    {\centerline{\textbf{Abstract}}\bigskip
      \begin{center}
       \begin{minipage}{11cm}
        \begin{small}
    }
    {   \end{small}
       \end{minipage}
      \end{center}
     \bigskip
    }

\numberwithin{equation}{section}
\numberwithin{figure}{section}
\newcounter{example}[section]
\numberwithin{example}{section}
\newcounter{remark}[section]
\numberwithin{remark}{section}
\newtheorem{theorem}{Theorem}[section]

\newtheorem{lemma}[theorem]{Lemma}
\newtheorem{corollary}[theorem]{Corollary}

\newtheorem{exmp}[example]{Example}
\newtheorem{rmrk}[remark]{Remark}
\newenvironment{example}{\begin{exmp}\rm}{\end{exmp}}
\newenvironment{remark}{\begin{rmrk}\rm}{\end{rmrk}}

\newcommand{\citeasnoun}[1]{\cite{#1}}

\begin{document}
\thispagestyle{empty}
\ifims{
\title{\thetitle}
\author{
 Spokoiny, Vladimir
\\[5.pt]
{
Weierstrass-Institute,} \\
{ Mohrenstr. 39,} 
10117 Berlin, Germany \\
\texttt{ spokoiny@wias-berlin.de} }
\maketitle
\pagestyle{myheadings} 
\markboth
 {\hfill {\sc \small \theruntitle} \hfill}
 {\hfill {\sc \small spokoiny, v.} \hfill}
\begin{abstract}
\theabstract
\end{abstract}

} 
{ 
\begin{frontmatter}
\title{\thetitle}


\runtitle{\theruntitle}

\begin{aug}
\author{\fnms{ Vladimir} \snm{Spokoiny}\ead[label=e1]{spokoiny@wias-berlin.de}}
\address{Weierstrass-Institute and\\
  Humboldt University Berlin, \\
 Mohrenstr. 39, 10117 Berlin, Germany \\
 \printead{e1}}
 \end{aug}

 \runauthor{spokoiny, v. }
\affiliation{Weierstrass-Institute and  Humboldt University Berlin}

\begin{abstract}
\theabstract
\end{abstract}



\end{frontmatter}
} 

\def\zv{\bb{z}}

\section{Introduction}
One of the most popular approaches in statistics is based on the parametric assumption
that the distribution \( \P \) of the observed data
\( \Yn \) belongs to a given parametric family \( (\P_{\thetav},\, \thetav \in \Theta) \), where \( \Theta \)
is a subset  in a finite dimensional space \(  \R^{p} \). In this
situation, the statistical inference about \( \P \) is reduced to recovering  \( \thetav \).
The standard likelihood  principle suggests to estimate  \( \thetav \)
by maximizing the corresponding log-likelihood function \( L(\Yn,\thetav) \).
The classical parametric statistical theory focuses mostly  on asymptotic properties of the
difference between  \( \tilde{\thetav}\) and the true value \( \thetavs \) as
the sample size \( n \) tends to infinity. 
There is a vast literature on this issue. 
We only mention the book by \citeasnoun{IH1981}, which provides a
comprehensive study of asymptotic properties of maximum likelihood and Bayesian
estimators.
The related analysis is effectively based on the Taylor expansion of the likelihood 
function near the true point under the assumption that the considered estimate 
well concentrates in a small (root-n) neighborhood of this point. 
In the contrary, there is only few results which establish this desired 
concentration property. \citeasnoun{IH1981}, Section 1.5, presents some 
exponential concentration bounds in the i.i.d. parametric case. 
Large deviation results about minimum  contrast estimators can be found in
\citeasnoun{J1998} and
\citeasnoun{sieddzh1987}, 
while subtle small sample size  properties of these estimators are presented in
\citeasnoun{F1982} and
\citeasnoun{F1990}.
This paper aims at studying the concentration properties of a general 
parametric estimate. The main result describes some concentration sets for the 
considered estimate and establishes an exponential bound for deviating of 
the estimate out of such sets. 

In the modern statistical literature there is a number of papers considering 
maximum likelihood or more generally
minimum contrast estimators in a general i.i.d.
situation, when the parameter set \(\Theta\) is a subset of some functional space. We
mention the papers
\citeasnoun{vdG1993},
\citeasnoun{BiMa1993},
\citeasnoun{BiMa1998},
\citeasnoun{Bi2006}
and references therein.
The studies mostly focused on the concentration properties of the maximum over
\( \thetav \in \Theta \) of the log-likelihood \( L(\Yn,\thetav) \)
rather on the properties of the estimator \( \tilde{\thetav} \)
which is the point of maximum of \( L(\Yn,\thetav) \).
The established results are based on deep probabilistic facts from
the empirical process theory (see e.g.
\citeasnoun{Ta1996},
\citeasnoun{VW1996},
\citeasnoun{BuLuMa2003}).
Our approach is similar in the sense that 
the analysis also focuses on the
properties of the maximum of \( L(\Yn,\thetav) \) over \( \thetav \in \Theta \).
However, we do not assume any specific structure of the model.
In particular, we do not assume independent
observations and thus, cannot apply the methods from the empirical process theory.

The aim of this paper is to offer a general and unified approach to statistical
estimation problem which delivers meaningful and informative results in a
general framework under mild regularity assumptions.
An important issue of the proposed approach is that it allows to go beyond the
parametric case, that is, the most of results and conclusions continue to apply
even if the parametric assumption is not precisely fulfilled.
Then the target of estimation can be viewed as the best parametric fit.
Some other important features of the proposed approach are that
 the established risk bounds are nonasymptotic and equally apply to large, moderate and
small samples and that the results describe nonasymptotic confidence and concentration sets in terms of
quasi log-likelihood rather than the accuracy of point estimation.
In the most of examples, the usual consistency and rate results can be easily obtained as
corollaries from the established risk bounds.
The results are obtained under very mild conditions which are easy to verify
in particular applications. There is no specific assumptions on the structure
of the data like independence or weak dependence of observations, the parameter
set can be unbounded. Another interesting feature of the proposed approach that
it does not require any identifiability conditions. 
Informally, one can say that 
whatever the quasi likelihood or contrast is, the corresponding estimate 
belongs with a dominating probability to the corresponding concentration set.
Examples show that the resulting concentration sets are of right magnitude, 
in typical situations this is a root-n vicinity of the true point.

Now we specify the considered set-up.
Let \( \Yn \) stand for the observed data.
For notational simplicity we assume that \( \Yn \) is a vector in \( \R^{n} \).
By \( \P \) we denote the measure describing the distribution of the whole
sample \( \Yn \).
The parametric approach discussed below allows to reduce the whole description
of the model to a few parameters which have to be estimated from the data.
Let \( \bigl( \P_{\thetav}, \thetav \in \Theta \bigr) \) be a given parametric
family of measures on \( \R^{n} \). The parametric assumption means simply that
\( \P = \P_{\thetavs} \) for some \( \thetavs \in \Theta \).
The parameter vector
\( \thetavs \) can be estimated using the maximum likelihood (MLE) approach.
Let \( L(\Yn,\thetav) \) be
the log-likelihood for the considered parametric model:
\( L(\Yn,\thetav) = \log \frac{d\P_{\thetav}}{d\P_{0}}(\Yn) \), where \( \P_{0} \)
is any dominating measure for the family \( (\P_{\thetav}) \).
The MLE estimate \( \tilde{\thetav} \) of the parameter
\( \thetavs \) is given by maximizing the log-likelihood \( L(\thetav) \):
\begin{eqnarray}
\label{hatthetaI}
    \tilde{\thetav} = \argmax_{\thetav \in \Theta} L(\Yn,\thetav).
\end{eqnarray}
Note that the value of the estimate will not be changed if the process 
\( L(\Yn,\thetav) \) is multiplied by any positive constant \( \mu \).

The quasi maximum likelihood approach admits that 
the underlying distribution \( \P \) does not belong to the family \( (\P_{\thetav}) \).
The estimate \( \tilde{\thetav} \) from (\ref{hatthetaI}) is still meaningful
and it becomes the \emph{quasi} MLE.
Later we show that \( \tilde{\thetav} \) estimate the value \( \thetavs \)
defined by maximizing the expected value of \( L(\Yn,\thetav) \):
\begin{eqnarray*}
    \thetavs
    \eqdef
    \argmax_{\thetav \in \Theta} \E L(\Yn,\thetav)
\end{eqnarray*}
which is the true value in the parametric situation and
can be viewed as the parameter of the best parametric fit in the general case.

Note that the presented set-up is quite general and the most of statistical
estimation procedures can be represented as quasi maximum likelihood for a
properly selected parametric family.
In particular, popular least squares, least absolute deviations,
M-estimates can be represented as quasi MLE.

The set-up of this paper is even more general. Namely, we consider a general 
estimate \( \tilde{\thetav} \) defined by maximizing a random field 
\( \LL(\thetav) \). 
The basic example we have in mind is the scaled quasi log-likelihood 
\( \LL(\thetav) = \mu L(\Yn,\thetav) \) for some \( \mu > 0 \). 
In some cases, especially if the parameter set is unbounded, the scaling factor 
\( \mu \) can also be taken depending on \( \thetav \), that is,
\( \LL(\thetav) = \mu(\thetav) L(\Yn,\thetav) \).
We focus on the properties of the process
\( \LL(\thetav) \) as a function of the parameter \( \thetav \). Therefore,
we suppress the argument \( \Yn \) there. 
One has to keep in mind that \( \LL(\thetav) \) is random
and depends on the observed data \( \Yn \).
The study focuses on the concentration properties of the estimate 
\( \tilde{\thetav} \) which is defined by maximization of the random process
\( \LL(\thetav) \). 
Let 
\begin{eqnarray*}
    \thetavs = \argmax_{\thetav} \E \LL(\thetav) .
\end{eqnarray*}
We also define \( \LL(\thetav,\thetavs) = \LL(\thetav) - \LL(\thetavs) \).
The aim of our study is to bound the value of the quasi maximum likelihood
\(
    \LL(\tilde{\thetav},\thetavs)
    =
    \max_{\thetav} \LL(\thetav,\thetavs)
\).
The basic assumption imposed on the process \( \LL(\thetav) \) is that the
difference \( \LL(\thetav,\thetavs) = \LL(\thetav) - \LL(\thetavs) \) has bounded
exponential moments for every \( \thetav \).
Our primary goal is to bound the supremum of such differences, or more
precisely, to establish an exponential bound for the value
\( \LL(\tilde{\thetav},\thetavs) \).
The standard approach of empirical process theory is to consider separately the
mean and the centered stochastic deviations of the process \( \LL(\thetav) \).
Here a slightly different standardization of the process \( \LL(\thetav) \) is used.
Assume that 
the exponential moment for
\( \LL(\thetav,\thetavs) \) is finite for all \( \thetav \).
This enables us to define for each \( \thetav \)
the \emph{rate function} \( \Lmgf(\thetav,\thetavs) \) which ensures the identity
\begin{eqnarray*}
    \E \exp\bigl\{ \LL(\thetav,\thetavs) + \Lmgf(\thetav,\thetavs) \bigr\}
    =
    1.
\end{eqnarray*}
This means that the process \( \LL(\thetav,\thetavs) + \Lmgf(\thetav,\thetavs) \)
is pointwise stochastically bounded in a rather strict sense.
We aim at establishing a similar bound for the maximum
of \( L(\thetav,\thetavs) + \Lmgf(\thetav,\thetavs) \). It turns out
that some payment for taking the maximum is necessary. Namely, we present a
penalty function \( \pen(\thetav) \) which ensures that the maximum of
\( \LL(\thetav,\thetavs) + \Lmgf(\thetav,\thetavs) - \pen(\thetav) \)
is bounded with exponential moments. Then we show that this fundamental fact
yields a number of straightforward corollaries about the quality of estimation.

The paper is organized as follows. The next section presents the main result
which describes an exponential upper bound for the (quasi) maximum likelihood.
Section~\ref{Scorolexp} discusses some implications of this exponential bound
for statistical inference.
In particular, we present a general likelihood-based construction of confidence sets
and establish an exponential bound for the coverage probability.
We also show that the considered estimate well concentrates on the level set of
the rate function \( \Lmgf(\thetav,\thetavs) \).
Under some standard conditions we show that such concentration sets become
usual root-n neighborhoods of the target \( \thetavs \).
Sections~\ref{SexpGLM} and \ref{SexpSI} illustrate the obtained general results
for two quite popular statistical models: generalized linear and single index.
These models are very well studied, the existing results claim asymptotic
normality and efficiency of the maximum likelihood estimate as the sample size
grows to infinity. On the contrary, our study focuses on nonasymptotic
deviation bounds and concentration properties of this estimate.
The main result giving an exponential bound for the maximum likelihood is based
on general results for the maximum of a random field described in
Section~\ref{SriskboundG}.

\section{Exponential bound for the maximum likelihood}
\label{Sriskboundex}
This section presents a general exponential bound on the
(quasi) maximum likelihood value in a quite general set-up.
The main result concerns the value of maximum
\( \LL(\tilde{\thetav}) = \max_{\thetav \in \Theta} \LL(\thetav) \)
rather than the point of maximum \( \tilde{\thetav} \).
Namely, we  aim at establishing some exponential bounds on the
supremum in \( \thetav \) of the random field
\begin{eqnarray*}
    \LL(\thetav,\thetavs)
    \eqdef
    \LL(\thetav) - \LL(\thetavs).
\end{eqnarray*}

%
In this paper we do not specify the structure of the process \( \LL(\thetav) \).
The basic assumption we impose on the considered model is that
\( \LL(\thetav) \) is absolutely continuous in \( \thetav \) and that
\( \LL(\thetav) \)  and its gradient w.r.t. \( \thetav \) have bounded exponential
moments.

\begin{description}
\item[\( \bb{(E)} \)]
\emph{
The rate function \( \Lmgf(\thetav,\thetavs) \) is finite for all 
\( \thetav \in \Theta \):
}
\begin{eqnarray*}
    \Lmgf(\thetav,\thetavs)
    \eqdef
    - \log \E \exp \bigl\{ \LL(\thetav,\thetavs) \bigr\}.
\end{eqnarray*}
\end{description}
Note that this condition is automatically fulfilled if \( \P = \P_{\thetavs} \)
and \( \LL(\thetav) = \mu \log \bigl( d\P_{\thetav}/d\P_{\thetavs} \bigr) \)
with \( \mu \le 1 \) provided that
all \( \P_{\thetav} \) are absolutely continuous w.r.t. \( \P_{\thetavs} \).
With \( \mu = 1 \) and 
\( \LL(\thetav) = \log \bigl( d\P_{\thetav}/d\P_{\thetavs} \bigr) \), it holds
\( \Lmgf(\thetav,\thetavs) = - \log \E_{\thetavs} (d\P_{\thetav}/d\P_{\thetavs})
= 0 \). For \( \mu < 1 \),
\( \Lmgf(\thetav,\thetavs) = - \log \E_{\thetavs} (d\P_{\thetav}/d\P_{\thetavs})^{\mu} 
\ge 0 \) for \( \mu < 1 \) by the Jensen inequality.

The main observation behind the condition \( (E) \) is that
\begin{eqnarray*}
    \E \exp\bigl\{ \LL(\thetav,\thetavs) + \Lmgf(\thetav,\thetavs) \bigr\}
    =
    1.
\end{eqnarray*}
Our main goal is to get a similar bound for the
maximum of the random field
\( \LL(\thetav,\thetavs) + \Lmgf(\thetav,\thetavs) \)
over \( \thetav \in \Theta \).
Below in Section~\ref{Scorolexp} we show that such a bound implies
an exponential bound for the coverage probability for a confidence set
\( \cc{E}(\zz) = \{ \thetav: \LL(\tilde{\thetav},\thetav) \le \zz \} \) and
that the estimate \( \tilde{\thetav} \) well concentrates on
a set
\( \cc{A}(\rr,\thetavs) = \{ \thetav: \Lmgf(\thetav,\thetavs) \le \rr \} \)
in the sense that the probability of the event
\( \{ \tilde{\thetav} \not\in \cc{A}(\rr,\thetavs) \} \) is exponentially small in
\( \rr \).

Unfortunately, in some situations, the exponential moment of the maximum
of \( \LL(\thetav,\thetavs) + \Lmgf(\thetav,\thetavs) \) can be
unbounded.
We present a simple example of this sort.
\begin{example}
\label{ExGausshift}
Consider a Gaussian shift with only one observation \( Y \sim \cc{N}(\theta,1) \)
and suppose that the true parameter is \( \thetas = 0 \). 
Then the log-likelihood ratio \( \LL(\theta,\thetas) \) reads as 
\( \LL(\theta,\thetas) = Y \theta - \theta^{2}/2 \), and
it holds \( \Lmgf(\theta,\thetas) = 0 \),
\( \sup_{\theta} \LL(\theta) = Y^{2}/2 \) and
\( \E_{\thetas} \exp\bigl\{ \sup_{\theta} \LL(\theta,\thetas) \bigr\}
= \E_{\thetas} \exp\bigl\{ Y^{2}/2 \bigr\} = \infty \).
\end{example}

We therefore consider the penalized expression
\( \LL(\thetav,\thetavs) + \Lmgf(\thetav,\thetavs) - \pen(\thetav) \),
where the penalty function \( \pen(\thetav) \) should provide some
bounded exponential moments for
\begin{eqnarray*}
    \sup_{\thetav \in \Theta}
    \bigl[ \LL(\thetav,\thetavs) + \Lmgf(\thetav,\thetavs) - \pen(\thetav) \bigr] .
\end{eqnarray*}
To bound local fluctuations of the process \( \LL(\thetav) \), we introduce an
exponential moment condition on the stochastic component \( \zeta(\thetav) \):
\begin{eqnarray*}
    \zeta(\thetav)
    \eqdef
    \LL(\thetav) - \E \LL(\thetav).
\end{eqnarray*}
Suppose also that the random function \( \zeta(\thetav) \) is differentiable
in \( \thetav \) and its gradient
\( \nabla \zeta(\thetav) = \partial \zeta(\thetav)/ \partial \thetav \in \R^{p} \)
fulfills the following condition:
\begin{description}
\item[\( \bb{(E\!D)} \)]
\emph{
There exist some continuous
symmetric matrix function \( V(\thetav) \) for \( \thetav \in\Theta \) and
constant \( \lambdab > 0 \) such that for all \( |\lambda| \le \lambdab \)
}
\begin{eqnarray}
\label{expzetac}
    \sup_{\gammav \in \cc{S}^{p}}
    \sup_{\thetav \in \Theta}
    \log \E \exp\biggl\{
        2 \lambda \frac{\gammav^{\T} \nabla \zeta(\thetav)}{\sqrt{\gammav^{\T} V(\thetav) \gammav}}
    \biggr\}
    \le
    2 \lambda^{2} .
\end{eqnarray}
\end{description}

Define for every \( \thetav,\thetavc \in \Theta \),
\( \normc = \| \thetav - \thetavc \| \) and
\( \gammav = (\thetavc - \thetav)/\normc \)
\begin{eqnarray*}
    \Lexpm^{2}(\thetav,\thetavc)
    \eqdef
    \normc^{2}
    \int_{0}^{1} \gammav^{\T} V(\thetav + t \normc \gammav) \gammav dt .
\end{eqnarray*}
Next, introduce
for every \( \thetavd \in \Theta \) the local vicinity
\( \B(\reps,\thetavd) \) such that \( \Lexpm(\thetav,\thetavd) \le \reps \)
for all \( \thetav \in \B(\reps,\thetavd) \).

Let also the function \( V(\cdot) \) from \( (E\!D) \) satisfy the following
regularity condition:
\begin{description}
\item[\( \bb{(V)} \)]
\textit{
There exist constants
\( \reps > 0 \) and \( \nu_{1} \ge 1 \) such that
\begin{eqnarray*}
    \sup_{\thetav,\thetavd \in \Theta: \,\, \Lexpm(\thetav,\thetavd) \le \reps}
    \,\,
    \sup_{\gammav \in S^{p}}
    \frac{\gammav^{\T} V(\thetav) \gammav}{\gammav^{\T} V(\thetavd) \gammav}
    \le \nu_{1} \, .
\end{eqnarray*}
}
\end{description}

Now we are prepared to state the main result which gives some sufficient 
condition on the penalty function \( \pen(\thetav) \) ensuring the desired 
penalized exponential bound. It is a specification of a more general result 
from Theorem~\ref{Tsmoothpen} in Section~\ref{SriskboundG}. 

Here and in what follows \( \omega_{p} \) (resp. \( \entr_{p} \))
denotes the volume (resp. the entropy number) of the unit ball in \( \R^{p} \).

\begin{theorem}
\label{Tmainbound}
Suppose that the conditions \( (E) \) is fulfilled and
\( (E\!D) \) holds with some \( \lambdab \) and a matrix function
\( V(\thetav) \) which fulfills \( (V) \) for \( \reps > 0 \) and
\( \nu_{1} \ge 1 \).
If for some \( \mrho \in (0,1) \) with
\( { \mrho \reps}/(1-\mrho) \le \lambdab \),  the penalty function
\( \pen(\thetav) \) fulfills
\begin{eqnarray}
\label{penaltycondition}
    \penH_{\reps}(\mrho)
    \eqdef
    \log \biggl\{ \omega_{p}^{-1} \reps^{-p}
        \int_{\Theta} \sqrt{\det(V(\thetav))}
        \exp\bigl\{ - \mrho \pen_{\reps}(\thetav) \bigr\}  d \thetav
    \biggr\}
    < \infty
\end{eqnarray}
with \( \pen_{\reps}(\thetavd)
    =
    \inf_{\thetav \in \B(\reps,\thetavd)} \pen(\thetav) \),
then
\begin{eqnarray}
\label{penalizedbound}
    \E \exp\Bigl\{
        \sup_{\thetav \in \Theta} \mrho \bigl[
                \LL(\thetav,\thetavs)
                + \Lmgf(\thetav,\thetavs)
                - \pen(\thetav)
        \bigr]
    \Bigr\}
    \le
    \Crlq(\mrho)
\end{eqnarray}
where
\begin{eqnarray}
\label{Crlqrho}
    \log \Crlq(\mrho)
    =
    \frac{2 \reps^{2} \mrho^{2}}{1 - \mrho}
    + (1 - \mrho) \entr_{p}
    + \penH_{\reps}(\mrho)
    + p \log (\nu_{1}).
\end{eqnarray}
\end{theorem}

\subsection{Penalty via the norm \( \| \sqrt{V^{*}}(\thetav - \thetavs) \| \)}
The choice of the penalty function \( \pen(\thetav) \) can be made more precise 
if \( V(\thetav) \le V^{*} \) for a fixed matrix \( V^{*} \) and all 
\( \thetav \). 
This section describes how the penalty function can be defined 
in terms of the norm \( \| \sqrt{V^{*}}(\thetav - \thetavs) \| \). 

\begin{theorem}
\label{Tmainboundranking}
Let the conditions \( (E) \) and \( (E\!D) \) be fulfilled
and in addition \( V(\thetav) \le V^{*} \) for some matrix \( V^{*} \) for all
\( \thetav \in \Theta \). Let
\( \mrho \in (0,1) \) and \( \reps > 0 \) be fixed to ensure 
\( { \mrho \reps}/(1-\mrho) \le \lambdab \).
Suppose that \( \penb(\rr) \) is a monotonously decreasing positive 
function on \( [0,+\infty ) \) satisfying 
\begin{eqnarray}
\label{penIntdef}
    \penInt^{*}
    \eqdef
    \omega_{p}^{-1} \int_{\R^{p}} \penb(\| \zv \|) d\zv 
    = 
    p \int_{0}^{\infty}  \penb(t) t^{p-1} dt
    < \infty .
\end{eqnarray}
Define 
\begin{eqnarray}
\label{penpenb}
    \pen(\thetav) 
    = 
    - \mrho^{-1} \log \penb\bigl( 
        \reps^{-1} \bigl\| \sqrt{V^{*}} (\thetav - \thetavs) \bigr\| + 1 
    \bigr).
\end{eqnarray}    
Then the assertion (\ref{penalizedbound}) holds with
\begin{eqnarray*}
    \log \Crlq(\mrho)
    =
    \frac{2 \reps^{2} \mrho^{2}}{1 - \mrho}
    + (1 - \mrho) \entr_{p}
    + \log (\penInt^{*}) .
\end{eqnarray*}
\end{theorem}

\begin{proof}
This result is a straightforward corollary of Theorem~\ref{Tmainbound} applied with
\( V(\thetav) \equiv V^{*} \) and thus, condition \( (V) \) is fulfilled with 
\( \nu_{1} = 1 \).
\end{proof}

Here two natural ways of defining the penalty function 
\( \pen(\thetav) \): quadratic or logarithmic in 
\( \bigl\| \sqrt{V^{*}} (\thetav - \thetavs) \bigr\| \).
The functions \( \penb(\cdot) \)
and the corresponding \( \penInt^{*} \)-values are:
\begin{eqnarray}
\label{penint12}
\begin{array}{rclrcl}
    \penb_{1}(t)
    &=&
    e^{ - \delta_{1} (t-1)_{+}^{2}},
    &
    \penInt^{*}_{1}
    &=&
    1 + \omega_{p}^{-1} (\pi /\delta_{1})^{p/2},    
    \\
    \penb_{2}(t) 
    &=& 
    (t + 1)^{- p - \delta_{2}},
    &
    \penInt^{*}_{2}
    &=&
    p /\delta_{2} \, ,
\end{array}  
\end{eqnarray}
where \( \delta_{1},\delta_{2} > 0 \) are some constant and \( [a]_{+} \)
means \( \max\{ a,0 \} \). 
The corresponding penalties read as:
\begin{eqnarray*}
    \pen_{1}(\thetav) 
    &=& 
    \mrho^{-1} \delta_{1} \,  
    \reps^{-2} \bigl\| \sqrt{V^{*}} (\thetav - \thetavs) \bigr\|^{2} .
    \nn
    \pen_{2}(\thetav) 
    &=& 
    - \mrho^{-1} (p + \delta_{2}) 
    \log \bigl( \reps^{-1} \bigl\| \sqrt{V^{*}} (\thetav - \thetavs) \bigr\| + 2 \bigr).
\end{eqnarray*}

\subsection{Some corollaries}
\label{Scorolexp}
Theorem~\ref{Tmainbound} claims that the value
\(
    \LL(\thetav,\thetavs) + \Lmgf(\thetav,\thetavs) - \pen(\thetav)
\)
is uniformly in \( \thetav \in \Theta \) stochastically bounded.
In particular, one can plug the estimate \( \tilde{\thetav} \) in place of
\( \thetav \):
\begin{eqnarray}
\label{expbound}
    \E \exp\Bigl\{
        \mrho \bigl[
            \LL(\tilde{\thetav},\thetavs)
            + \Lmgf(\tilde{\thetav},\thetavs)
            - \pen(\tilde{\thetav})
        \bigr]
    \Bigr\}
    \le
    \Crlq(\mrho) .
\end{eqnarray}
Below we present some corollaries of this result.

\subsubsection{Concentration properties of the estimator \( \tilde{\thetav} \)}
Define for every subset \( A \) of the parameter set \( \Theta \) the value
\begin{eqnarray}
\label{Azex}
    \zz(A)
    \eqdef
    \inf_{\thetav \not\in A}
    \{ \Lmgf(\thetav,\thetavs) - \pen(\thetav) \}.
\end{eqnarray}    
The next result shows that the estimator \( \tilde{\thetav} \) deviates
out of the set \( A \) with an  exponentially small probability of order 
\( \exp\{  - \mrho \zz(A) \} \).

\begin{corollary}
\label{CLDboundR0}
Suppose (\ref{expbound}). 
Then for any set \( A \subset \Theta \)
\begin{eqnarray*}
    \P\bigl( \tilde{\thetav} \not \in A \bigr)
    \le 
    \Crlq(\mrho) \ex^{- \mrho \zz(A)}.
\end{eqnarray*}
\end{corollary}

\begin{proof}
If \( \tilde{\thetav} \not\in A \), then 
\( \Lmgf(\tilde{\thetav},\thetavs) - \pen(\tilde{\thetav}) \ge \zz(A) \).
As \( \LL(\tilde{\thetav},\thetavs) \ge 0 \), it follows
\begin{eqnarray*}
    \Crlq(\mrho)
    & \ge &
    \E \exp\Bigl\{
        \mrho \bigl[
            \LL(\tilde{\thetav},\thetavs)
            + \Lmgf(\tilde{\thetav},\thetavs)
            - \pen(\tilde{\thetav})
        \bigr]
    \Bigr\}
    \nn
    & \ge &
    \E \exp\Bigl\{
        \mrho \bigl[
            \Lmgf(\tilde{\thetav},\thetavs)
            - \pen(\tilde{\thetav})
        \bigr]
    \Bigr\}
    \ge 
    \ex^{\zz(A)} \P\bigl( \tilde{\thetav} \not\in A \bigr)
\end{eqnarray*}    
as required.
\end{proof}

Two particular choices of the set \( A \) can be mentioned:
\begin{eqnarray*}
    A 
    &=& 
    \cc{A}(\rr,\thetavs) = \{ \thetav: \Lmgf(\thetav,\thetavs) \le \rr \} ,
    \nn
    A 
    &=& 
    \cc{A}'(\rr,\thetavs) 
    = 
    \{ \thetav: \Lmgf(\thetav,\thetavs) - \pen(\thetav) \le \rr \} ,
\end{eqnarray*}
For the set \( \cc{A}'(\rr,\thetavs) \), Corollary~\ref{CLDboundR0} yields
\begin{eqnarray*}
    \P\bigl( \tilde{\thetav} \not \in \cc{A}'(\rr,\thetavs) \bigr)
    =
    \P\bigl( \Lmgf(\tilde{\thetav},\thetavs) - \pen(\tilde{\thetav}) \ge \rr \bigr)
    \le 
    \Crlq(\mrho) \ex^{- \mrho \rr}.
\end{eqnarray*}
For the set \( \cc{A}(\rr,\thetavs) \), define additionally \( \pnn(\rr) \) as the 
minimal value for which 
\begin{eqnarray*}
    \Lmgf(\thetav,\thetavs) - \pen(\thetav)
    \ge 
    \rr - \pnn(\rr),
    \qquad 
    \thetav \not\in \cc{A}(\rr,\thetavs),
\end{eqnarray*}    
or, equivalently,
\begin{eqnarray}
\label{pnnrr}
    \pnn(\rr)
    =
    \sup_{\thetav \not\in \cc{A}(\rr,\thetavs)} 
    \bigl\{ \rr + \pen(\thetav) - \Lmgf(\thetav,\thetavs) \bigr\}.
\end{eqnarray}    

\begin{corollary}
\label{CLDboundR}
Suppose (\ref{expbound}). 
Then for any \( \rr > 0 \)
\begin{eqnarray*}
    \P\bigl( \tilde{\thetav} \not \in \cc{A}(\rr,\thetavs) \bigr)
    =
    \P\bigl( \Lmgf(\tilde{\thetav},\thetavs) \ge \rr \bigr)
    \le 
    \Crlq(\mrho) \ex^{- \mrho [\rr - \pnn(\rr)]}.
\end{eqnarray*}
\end{corollary}

In typical situations the value \( \Lmgf(\thetav,\thetavs) \) is nearly proportional to the sample size
\( n \) and is nearly quadratic in \( \thetav - \thetavs \) so that for a fixed
\( \rr \) the set \( \cc{A}(\rr,\thetavs) \) corresponds to a root-\( n \)
neighborhood of the point \( \thetavs \).
See below in Section~\ref{Sergodic} for a precise formulation.

\subsubsection{Confidence sets based on \( \LL(\tilde{\thetav},\thetav) \)}
Next we discuss how the exponential bound can be used for establishing some
risk bounds and for constructing the confidence sets for the target
\( \thetavs \) based on the maximized value \( \LL(\tilde{\thetav},\thetav) \).
The inequality (\ref{expbound}) claims that  \( \LL(\tilde{\thetav},\thetavs) \)
is stochastically bounded with finite exponential moments.
This implies boundness of the polynomial moments.

Define
\begin{eqnarray}
\label{pnnex}
    \pnn
    \eqdef
    \pnn(0)
    =
    \sup_{\thetav} [\pen(\thetav) - \Lmgf(\thetav,\thetavs)]_{+} \, .
\end{eqnarray}

\begin{corollary}
\label{CLDboundRB}
Suppose (\ref{expbound}) and let
\( \pnn \) from (\ref{pnnex}) be finite. Then
\begin{eqnarray*}
    \E \exp\bigl\{
        \mrho \LL( \tilde{\thetav},\thetavs \bigr)
        \bigr\}
    \le
    e^{\mrho \pnn} \Crlq(\mrho).
\end{eqnarray*}
\end{corollary}

\begin{proof}
Observe that
\begin{eqnarray*}
    \E \exp\bigl\{
        \mrho \LL( \tilde{\thetav},\thetavs \bigr)
        \bigr\}
    \le 
    e^{\mrho \pnn}\E \exp\Bigl\{
        \mrho \bigl[
            \LL( \tilde{\thetav},\thetavs \bigr)
            + \Lmgf(\tilde{\thetav},\thetavs) - \pen(\tilde{\thetav})
        \bigr]
    \Bigr\}
    \le 
    e^{\mrho \pnn} \Crlq(\mrho)
\end{eqnarray*}
as required.
\end{proof}

By the same reasons, one can construct confidence sets based on
the (quasi) likelihood process.
Define
\begin{eqnarray*}
    \cc{E}(\zz)
    =
    \bigl\{ \thetav\in \Theta: \LL(\tilde{\thetav},\thetav) \le \zz \bigr\}.
\end{eqnarray*}
The bound for \( \LL(\tilde{\thetav},\thetavs) \) ensures that
the target \( \thetavs \) belongs to this set with a high probability provided
that \( \zz \) is large enough.
The next result claims that
\( \cc{E}(\zz) \) does not cover the true value \( \thetavs \) with a probability
which decreases exponentially in \( \zz \).

\begin{corollary}
\label{CLDboundCS}
Suppose (\ref{expbound}).
For any \( \zz > 0 \)
\begin{eqnarray*}
    \P\bigl( \thetavs \notin \cc{E}(\zz) \bigr)
    \le
    \Crlq(\mrho) \exp\bigl\{ - \mrho \zz + \mrho \pnn  \bigr\}.
\end{eqnarray*}
\end{corollary}
\begin{proof}
The bound (\ref{expbound}) implies for the event
\( \{ \thetavs \notin \cc{E}(\zz) \} = \{ \LL(\tilde{\thetav},\thetavs) > \zz \} \)
\begin{eqnarray*}
     \P\bigl\{  \thetavs \notin \cc{E}(\zz) \bigr\}
     & \le  &
     \P\bigl\{
        \mrho \bigl[ \LL(\tilde{\thetav},\thetavs)
            + \Lmgf(\tilde{\thetav},\thetavs) - \pen(\tilde{\thetav})
        \bigr]
     > \mrho \zz - \mrho \pnn \bigr\}
     \nn
     & \le &
     \exp\bigl\{-\mrho \zz + \mrho \pnn \bigr\}
     \E \exp\bigl\{ 
        \mrho \bigl[ 
            \LL(\tilde{\thetav},\thetavs)
            + \Lmgf(\tilde{\thetav},\thetavs) - \pen(\tilde{\thetav}) 
        \bigr]
    \bigr\}
     \nn
     & \le &
     \Crlq(\mrho)\exp\bigl\{ - \mrho \zz + \mrho \pnn  \bigr\}
\end{eqnarray*}
as required.
\end{proof}

\subsection{Identifiability condition}
Until this point no any identifiability condition on the model has been used, that is, the
presented results apply even for a very poor parametrization. Actually, a
particular parametrization of the parameter set plays no role as long as the
value of maximum is considered. If we want to derive any quantitative result on the point of
maximum \( \tilde{\thetav} \), then the parametrization matters and
an identifiability condition is really necessary.
Here we follow the usual path
by applying the quadratic lower bound for the rate function \( \Lmgf(\thetav,\thetavs) \)
in a vicinity of the point \( \thetavs \).
Suppose that the rate function
\( \Lmgf( \thetav,\thetavs) = - \log \E \exp \bigl\{ \LL(\thetav,\thetavs) \bigr\} \)
 is two times continuously differentiable in \( \thetav \).
Obviously \( \Lmgf(\thetavs,\thetavs) = 0 \) and simple algebra yields
for the gradient \( \nabla \Lmgf(\thetav,\thetavs) = d \Lmgf(\thetav,\thetavs) / d \thetav
\):
\begin{eqnarray*}
    \nabla \Lmgf(\thetav,\thetavs)|_{\thetav = \thetavs}
    =
    - \E \nabla \LL(\thetav)|_{\thetav = \thetavs}
    =
    - \nabla \E \LL(\thetavs)
    = 0
\end{eqnarray*}
because \( \thetavs \) is the point of maximum of \( \E \LL(\thetav) \).
The Taylor expansion of the second order in a vicinity of \( \thetavs \) yields
for all \( \thetav \) close to \( \thetavs \) the following approximation:
\begin{eqnarray*}
    \Lmgf(\thetav,\thetavs)
    \approx
    (\thetav - \thetavs)^{\T} I(\thetavs) (\thetav - \thetavs)/2
\end{eqnarray*}
with the matrix 
\( I(\thetavs) = \E \nabla^{2} \Lmgf(\thetav,\thetavs)|_{\thetav = \thetavs} \).
So, one can expect that the rate function \( \Lmgf(\thetav,\thetavs) \) is 
nearly quadratic in \( \thetav - \thetavs \) in a neighborhood 
of the point \( \thetavs \). 

\begin{corollary}
\label{CLDboundMT}
Let (\ref{expbound}) hold.
Suppose that for some  positive symmetric matrix \( D \)
and some \( \rr > 0 \), 
the function \( \Lmgf(\thetav,\thetavs) \) fulfills
\begin{eqnarray}
\label{ratestt}
    \Lmgf(\thetav,\thetavs) 
    \ge
    (\thetav - \thetavs)^{\T} D^{2} (\thetav - \thetavs) ,
    \qquad
    \thetav \in \cc{A}(\rr,\thetavs),
\end{eqnarray}
Then for any \( \zz \le \rr \)
\begin{eqnarray*}
    \P\bigl( \| D(\tilde{\thetav} - \thetavs) \|^{2} > \zz \bigr)
    \le
    \Crlq(\mrho)
    \ex^{ - \mrho [\zz - \pnn(\zz)] } .
\end{eqnarray*}
\end{corollary}

\begin{proof}
It is obvious that
\begin{eqnarray*}
    \bigl\{ \| D (\tilde{\thetav} - \thetavs) \|^{2} > \zz \bigr\}
    & \subseteq &
    \bigl\{
        \| D (\tilde{\thetav} - \thetavs) \|^{2} > \zz , 
        \, \tilde{\thetav} \in \cc{A}(\rr,\thetavs)
    \bigr\}
    \cup
    \bigl\{  \tilde{\thetav} \not\in \cc{A}(\rr,\thetavs) \bigr\} 
    \nn
    & \subseteq &
    \bigl\{
        \Lmgf(\tilde{\thetav},\thetavs) > \zz , 
        \, \tilde{\thetav} \in \cc{A}(\rr,\thetavs)
    \bigr\}
    \cup
    \bigl\{  \Lmgf(\tilde{\thetav},\thetavs) > \zz \bigr\}
    \nn
    &=&
    \bigl\{  \Lmgf(\tilde{\thetav},\thetavs) > \zz \bigr\}
\end{eqnarray*}
and the result follows from Corollary~\ref{CLDboundR}.
\end{proof}

In the next theorem we assume the lower bound (\ref{ratestt}) to be fulfilled
on the whole parameter set \( \Theta \). The general case can be reduced to 
this one by using once again the concentration property of Corollary~\ref{CLDboundR}.

\begin{theorem}
\label{CLDboundMT2}
Suppose \( (E) \), \( (E\!D) \) with \( V(\thetav) \le V^{*} \) 
for a matrix \( V^{*} \). Let also for some \( \fis > 0 \)
\begin{eqnarray}
\label{Lmgfquad}
    \Lmgf(\thetav,\thetavs)
    \ge
    \fis^{2} (\thetav - \thetavs)^{\T} V^{*} (\thetav - \thetavs) ,
    \qquad
    \thetav \in \Theta .
\end{eqnarray}    
Fix some \( \fiss \le \fis \) and define
\( \pen(\thetav) 
\) by
\begin{eqnarray}
\label{pentid}
    \pen(\thetav)
    =
    \fiss^{2} (\thetav - \thetavs)^{\T} V^{*} (\thetav - \thetavs) .
\end{eqnarray}    
Then with  \( \srho = 1 - \fiss^{2} / \fis^{2} \) it holds
\begin{eqnarray}
\label{crlqb}
    \Crlq(\mrho,\srho)
    & \eqdef &
    \log \E \exp\bigl\{ \mrho \sup_{\thetav}\bigl[ \LL(\thetav,\thetavs) 
        + \Lmgf(\thetav,\thetavs) - \pen(\thetav) \bigr] 
    \bigr\}
    \nn
    & \le &
    2 \mrho
    + (1 - \mrho) \entr_{p}
    + \log \biggl( 
        1 
        + \frac{\omega_{p}^{-1} \pi^{p/2}}{(1 - \mrho)^{p/2} \fiss^{p}} 
    \biggr)
    \nn
    & \le &
    p C(\mrho)
    +
    p \log \bigl( |\fis^{2} (1 - \srho) (1 - \mrho)|^{-1/2} \bigr)
\end{eqnarray}    
for some fixed constant \( C(\mrho) \). 
In addition, \( \pnn(\rr) \) from (\ref{pnnrr}) fulfills \( \pnn(\rr)  = 0 \) for all 
\( \rr \ge 0 \) yielding 
for any \( \zz > 0 \) the concentration property and confidence bound:
\begin{eqnarray*}
\begin{array}{rclcrcl}
    \P\bigl( \tilde{\thetav} \not\in \cc{A}(\zz,\thetavs) \bigr)
    & \le &
    \Crlq(\mrho,\srho) \ex^{- \mrho \srho \zz  } ,
    & \quad &
    \cc{A}(\zz,\thetavs) 
    &=& 
    \{ \thetav: \Lmgf(\thetav,\thetavs) \le \zz \},
    \nn
    \P\bigl( \thetavs \not\in \cc{E}(\zz) \bigr)
    & \le &
    \Crlq(\mrho,0) \ex^{- \mrho \zz  } ,
    & \quad &
    \cc{E}(\zz) 
    &=& 
    \{ \thetav: \LL(\tilde{\thetav},\thetav) \le \zz \} .
\end{array}
\end{eqnarray*}    
\end{theorem}

\begin{proof}
We apply Theorem~\ref{Tmainboundranking} with 
\begin{eqnarray*}
    \penb(t)
    =
    \exp \bigl\{ - (1 - \mrho) \fiss^{2} (t-1)_{+}^{2} \bigr\} 
\end{eqnarray*}    
leading for \( \reps^{2} = (1 - \mrho) / \mrho \) and 
\( t = \reps^{-1} \bigl\| \sqrt{V^{*}} (\thetav - \thetavs) \bigr\| \)
to the formula (\ref{pentid}) for \( \pen(\thetav) \).
By simple algebra
\begin{eqnarray*}
    \penInt^{*}
    =
    \omega_{p}^{-1} \int_{\R^{p}} \penb(\| \thetav \|) d\thetav 
    = 
    1 +
    \omega_{p}^{-1} \frac{\pi^{p/2}}{(1 - \mrho)^{p/2} \fiss^{p}} ;
\end{eqnarray*}    
cf. the bound (\ref{penint12}) for \( \penInt^{*} \) 
with \( \delta_{1} = (1 - \mrho) \fiss^{2} \).
This implies the bound (\ref{crlqb}) for the \( \Crlq(\mrho) \)
because \( p^{-1} \entrl_{p} \) and \( p^{-1} \log \omega_{p}^{-1} \)
are bounded by some fixed constants.

The inequality (\ref{Lmgfquad}) ensures for \( \rr = \Lmgf(\thetav,\thetavs) \) 
that 
\(
    \pen(\thetav) 
    \le 
    \fiss^{2}/\fis^{2} \rr
\),
i.e. \( \pnn(\rr) \le \fiss^{2}/\fis^{2} \rr \) and \( \pnn = \pnn(0) = 0 \).
Finally, 
the concentration and coverage bounds follow from Corollaries~\ref{CLDboundR} and \ref{CLDboundCS}.
\end{proof}

\begin{remark}
\label{Rpenmix}
If the quadratic lower bound (\ref{Lmgfquad}) is only fulfilled for 
\( \thetav \) from an elliptic neighborhood \( \BB(\rr,\thetavs) 
= \bigl\{ \thetav: \bigl\| \sqrt{V^{*}} (\thetav - \thetavs) \bigr\| \le \rr \bigr\} \) of the point 
\( \thetavs \) with a sufficiently large \( \rr \), then it is reasonable to redefine the penalty function 
using the hybrid proposal: 
\begin{eqnarray*}
    \pen(\thetav)
    =
    \begin{cases}
        \fiss^{2} \bigl\| \sqrt{V^{*}} (\thetav - \thetavs) \bigr\|^{2}, 
        & \thetav \in \BB(\rr,\thetavs) , \\
        \mrho^{-1} (p + 1) \log \bigl( \bigl\| \sqrt{V^{*}} (\thetav - \thetavs) \bigr\| + 2 \bigr), 
        & \thetav \not\in \BB(\rr,\thetavs) .
    \end{cases}
\end{eqnarray*}    
Then the bound (\ref{crlqb}) still applies with the obvious correction of the value 
\( \Crlq(\mrho,\srho) \). However, the values \( \pnn \) and \( \pnn(\rr) \) from 
(\ref{pnnrr}) entering in our risk bounds have to be 
corrected depending on the behavior of the rate function \( \Lmgf(\thetav,\thetavs) \) 
for \( \thetav \not\in \BB(\rr,\thetavs) \).
\end{remark}

\subsection{Discussion}
This section collects some comments about the presented exponential bound.

\subsubsection*{Bounds for polynomial loss}
Our concentration result is stated in terms if the rate function 
\( \Lmgf(\thetav,\thetavs) \). Note that
the bounds (\ref{crlqb}) and (\ref{Lmgfquad}) imply the usual result about the 
quadratic loss \( \| \sqrt{V^{*}}(\tilde{\thetav} - \thetavs) \|^{2} \):
\begin{eqnarray*}
    \P\bigl( 
        \| \fis \sqrt{V^{*}}(\tilde{\thetav} - \thetavs) \|^{2} > \zz 
    \bigr) 
    \le 
    \Crlq(\mrho,\srho) \ex^{- \mrho \srho \zz  } .
\end{eqnarray*}
Note however, that the 
result (\ref{crlqb}) in terms of the rate function \( \Lmgf(\thetav,\thetavs) \) is 
more accurate because the lower bound (\ref{Lmgfquad}) can be very rough. 
The bound (\ref{Lmgfquad}) as well as the bound 
\( V(\thetav) \le V^{*} \)  are only used to evaluate the constants in the 
exponential risk bound. Moreover, if \( \mrho \) or \( \srho \) approaches one, 
the leading term in the risk bound is 
\( p \log \bigl( |(1 - \srho) (1 - \mrho)|^{-1/2} \bigr) \) which does not 
depend on \( \fis \) or \( V^{*} \).

\subsubsection*{Coverage probability and risk bounds}
The result of Corollary~\ref{CLDboundRB} justifies the use of confidence set
\( \cc{E}(\zz) = \bigl\{ \thetav: \LL(\tilde{\thetav},\thetav) \le \zz \bigr\} \).
However, the bound for the  coverage probability given 
by this result is quite rough and cannot be used for practical purposes. One 
has to apply one or another resampling scheme to fix a proper value \( \zz(\alpha) \)
providing the prescribed coverage probability \( 1 - \alpha \).

The same remark applies to the result of 
Corollary~\ref{CLDboundMT}. All these bounds are deduced from rather rough exponential 
inequalities and constants shown there are not optimal. However, the concentration 
property enables us to apply the classical one-step improvement technique to 
build a new estimate which 
achieves the asymptotic efficiency bound.

\subsubsection*{Root-n consistency}
\label{Sergodic}
Suppose that there exists a constant \( n \) (usually this constant means the sample
size) such that the functions
\begin{eqnarray*}
    v(\thetav)
    \eqdef
    n^{-1} V(\thetav),
    \qquad
    \lmgf(\thetav,\thetavs)
    \eqdef
    n^{-1}\Lmgf(\thetav,\thetavs)
\end{eqnarray*}
are continuous and bounded on every compact set by constants which only
depend on this set.
In addition we assume similarly to (\ref{ratestt}) that for some
fixed symmetric positive matrix \( D_{1} \) and some \( \rr > 0 \),
it holds in the vicinity \( \cc{A}(\rr,\thetavs) \) of the point \( \thetavs \):
\begin{eqnarray}
\label{lmfgvt}
    \qquad
    \lmgf(\thetav,\thetavs)
    \ge
    (\thetav - \thetavs)^{\T} D_{1}^{2} (\thetav - \thetavs),
    \qquad
    v(\thetav)
    \le
    \fis^{-2} D_{1}^{2} \, .
\end{eqnarray}
Then \( \Lmgf(\thetav,\thetavs)
\ge n (\thetav - \thetavs)^{\T} D_{1}^{2} (\thetav - \thetavs) \) and the elliptic
set \( \cc{A}^{*}(\rr,\thetavs)
\eqdef \{ \thetav: (\thetav - \thetavs)^{\T} D_{1}^{2} (\thetav - \thetavs) \le \rr/n \} \)
is a root-n neighborhood of the point \( \thetavs \).
By Theorem~\ref{CLDboundMT2} the estimate \( \tilde{\thetav} \) deviates from
this neighborhood with probability which decreases exponentially with \( \rr
\): 
\begin{eqnarray*}
    \P\bigl( \| D_{1} (\tilde{\thetav} - \thetavs) \|^{2} > \rr/n \bigr)
    \le
    \Crlq(\mrho,\srho) \ex^{- \mrho \srho \rr } .
\end{eqnarray*}

\subsubsection*{Local approximation}
The standard asymptotic theory of parameter estimation heavily uses the idea of 
local approximation: the considered (quasi) log likelihood is approximated by 
the log-likelihood of another simpler model in the vicinity of the true point
yielding the local asymptotic equivalence of the original and the approximating 
model. The local asymptotic normality (LAN) condition is the most popular example 
of this approach; see \cite{IH1981}, Ch. 2,  for more details.
 A combination of this idea with the concentration property of 
Corollary~\ref{CLDboundR} can be used to derive sharp asymptotic risk bounds 
for the estimate \( \tilde{\thetav} \); see again \cite{IH1981}, Ch. 3. 
Similarly one can derive non asymptotic risk in the framework of this paper.
However, a 
precise formulation of the related results is to be given elsewhere.

\subsubsection*{Large and moderate deviation}
The obtained results can be used to derive large and moderate deviations
for the estimate \( \tilde{\thetav} \); cf. 
\cite{J1998},
\cite{sieddzh1987}.
Particularly, the deviation result from 
Corollary~\ref{CLDboundR} can be used to study the efficiency of the estimate 
\( \tilde{\thetav} \) in the Bahadur sense; see e.g. \cite{Ar2006} and reference 
therein.

\section{Estimation in a generalized linear model}
\label{SexpGLM}
In this section we illustrate the general results of Sections~\ref{Sriskboundex}
and \ref{Scorolexp} on the problem of estimating the parameter vector
in the so called generalized linear model.
Let \( \cc{P} \) be an exponential family with the canonical parametrization
(EFC) which means that the corresponding log-likelihood function can
be written in the form
\begin{eqnarray*}
    \ell(y,\upsilon)
    =
    y \upsilon - d(\upsilon) + \ell(y)
\end{eqnarray*}
where \( d(\cdot) \) is a given convex function; see \citeasnoun{greensil1994}.
The term \( \ell(y) \)
is unimportant and it cancels in the log-likelihood ratio.

Let \( \Yn = (Y_{1},\ldots,Y_{n}) \) be an observed sample.
A \emph{generalized linear} assumption means
that the \( Y_{i} \)'s are independent, the distribution of every \( Y_{i} \)
belongs to \( \cc{P} \) and the corresponding parameter linearly depends
on given feature vectors \( \Psi_{i} \):
\begin{eqnarray}
\label{modelrd}
    Y_{i}
    \sim
    P_{\Psi_{i}^{\T} \thetav}
\end{eqnarray}
%
To be more specific we consider a deterministic explanatory variables \( \Psi_{1},\ldots,\Psi_{n}
\). The case of a random design can be considered in the same way.

The parametric assumption (\ref{modelrd})
leads to the log-likelihood
\( L(\thetav) = \sum_{i} \ell(Y_{i},\Psi_{i}^{\T}\thetav) \):
\begin{eqnarray}
\label{Lthetaex}
    L(\thetav) = \sum_{i} \ell(Y_{i},\Psi_{i}^{\T}\thetav)
    =
    \sum_{i} \bigl\{
        Y_{i} \Psi_{i}^{\T}\thetav - d\bigl( \Psi_{i}^{\T}\thetav \bigr)
        +
        \ell(Y_{i})
    \bigr\} .
\end{eqnarray}
Asymptotic properties of the MLE \( \tilde{\thetav} = \argmax_{\thetav} L(\thetav) \)
are well studied. We refer to \citeasnoun{fahr1985}, \citeasnoun{lang1996},
\citeasnoun{chen1999} and the book \citeasnoun{mccu1989} for further references.
The results claim asymptotic consistency, normality and efficiency of the
estimate \( \tilde{\thetav} \).

Our approach is a bit different because we do not assume that the underlying model
follow (\ref{modelrd}). The observations \( Y_{i} \) are independent, otherwise
any particular structure is allowed. In particular, the distribution of
every \( Y_{i} \) does not necessarily belong to \( \cc{P} \).
The considered problem is the problem of the best parametric
approximation of the data distribution \( \P \) by the GLM's of the form
\( \prod_{i} P_{\Psi_{i}^{\T} \thetav} \).
\begin{example}[Mean regression]
\label{Exmeanregr}
The least squares estimate \( \tilde{\thetav} \) in the classical mean regression
 minimizes the sum of squared residuals:
\begin{eqnarray*}
    \tilde{\thetav}
    =
    \argmin_{\thetav} \sum_{i} \bigl( Y_{i} - \Psi_{i}^{\T} \thetav \bigr)^{2} .
\end{eqnarray*}
This estimate can be viewed as the quasi MLE for to the Gaussian
homogeneous errors. However, many of its properties continue to hold even if
the errors are not i.i.d. Gaussian. What we only need is the existence of
exponential moments of the errors.
\end{example}

\begin{example}[Poisson regression]
\label{Expoisregr}
Let \( Y_{i} \) be some nonnegative integers, observed at ``locations'' \( X_{i} \). 
Such data often appear in digital imaging, positron emission tomography, 
queueing and traffic theory, and many others. A natural way of modeling such 
data is to assume that every \( Y_{i} \) is Poissonian with the parameter 
which depends on the locations \( X_{i} \) through the regression vector 
\( \Psi_{i} \). A generalized linear model assumes that the canonical parameter 
of the underlying Poisson distribution of \( Y_{i} \) linearly depends on the vector \( \Psi_{i} \)
leading to the (quasi) MLE
\begin{eqnarray*}
    \tilde{\thetav}
    =
    \argmax_{\thetav}
    \sum_{i=1}^{n} \bigl\{ 
        Y_{i} \Psi_{i}^{\T}\thetav - \exp( \Psi_{i}^{\T}\theta) 
    \bigr\}.
\end{eqnarray*}    
For our further analysis we only require that every \( Y_{i} \) has a bounded 
exponential moment, see below the condition (\ref{expzetai}) for a precise formulation.
\end{example}

In the general situation, 
for some \( \mu > 0 \) which will be fixed later, define 
\begin{eqnarray*}
    \LL(\thetav)
    \eqdef
    \mu \sum_{i} \bigl\{
        Y_{i} \Psi_{i}^{\T}\thetav - d\bigl( \Psi_{i}^{\T}\thetav \bigr)
    \bigr\} .
\end{eqnarray*}    
The target \( \thetavs \) maximizes \( \E \LL(\thetav) \):
\begin{eqnarray}
\label{thetavsglm}
    \thetavs
    =
    \argmax_{\thetav} \E \LL(\thetav)
    =
    \argmax_{\thetav}
    \sum_{i} \bigl\{
        \mY_{i} \Psi_{i}^{\T}\thetav - d(\Psi_{i}^{\T}\thetav)
    \bigr\} ,
\end{eqnarray}
where \( \mY_{i} = \E Y_{i} \). This yields
\begin{eqnarray}
\label{extremeexp}
    \nabla \E \LL(\thetav) \big|_{\thetav = \thetavs}
    &=&
    \mu \sum_{i} \bigl\{
        \mY_{i} - \dot{d}(\Psi_{i}^{\T}\thetavs)
    \bigr\} \Psi_{i}
    =
    0,
\end{eqnarray}
%
Next,
for every \( \thetav \)
\begin{eqnarray*}
    \zeta(\thetav)
    & \eqdef &
    \LL(\thetav) - \E \LL(\thetav)
    =
    \mu \sum_{i}  (Y_{i} - \mY_{i}) \Psi_{i}^{\T}\thetav ,
    \nn
    \nabla \zeta(\thetav)
    &=&
    \nabla \zeta
    =
    \mu \sum_{i}  (Y_{i} - \mY_{i}) \Psi_{i},
\end{eqnarray*}
Let there exist  a positive value \( \lambdab_{1} \), and
for every \( i \le n \) the value \( \expzeta_{i} > 0 \) such that
\begin{eqnarray}
\label{expzetai}
    \log \E \exp\Bigl\{
        2 \lambda \frac{Y_{i} - \mY_{i}}{\expzeta_{i}}
    \Bigr\}
    \le
    2 \lambda^{2},
    \qquad
    |\lambda| \le \lambdab_{1} \, .
\end{eqnarray}
In the case of Gaussian errors, one can take
\( \expzeta_{i} = s_{i} \eqdef \E^{1/2} (Y_{i} - \mY_{i})^{2} \).
Define
\begin{eqnarray}
\label{Vexpglm}
    V_{1}
    \eqdef
    \sum_{i} \expzeta^{2}_{i} \, \Psi_{i} \Psi_{i}^{\T} ,
    \qquad 
    V \eqdef \mu^{2} V_{1}.
\end{eqnarray}
The matrix \( V \) is a symmetric and non-negative.
Denote \( c_{i}(\gammav) \eqdef \gammav^{\T} \Psi_{i} \expzeta_{i} (\gammav^{\T} V_{1} \gammav)^{-1/2} \)
for any \( \gammav \in S^{p} \). 
By definition, \( |c_{i}(\gammav)| \le 1 \), however, usually 
\( |c_{i}(\gammav)| \) is much smaller, of order \( 1/\sqrt{n} \).

\begin{lemma}
\label{Lexpglm}
Suppose (\ref{expzetai}).
If for some \( \mu > 0 \), it holds 
\( \mu |\Psi_{i}^{\T} (\thetav - \thetavs)| \expzeta_{i} \le \lambdab_{1} \) for
for all \( i \) and all \( \thetav \in \Theta \), then \( (E) \) is satisfied.
Moreover, if \( \lambda \) is such that
\( |\lambda c_{i}(\gammav)| \le  \lambdab_{1} \)
for all \( i \) and any \( \gammav \in S^{d} \),
then \( (E\!D) \) holds with this \( \lambda \) and \( V(\cdot) \equiv V \)
from (\ref{Vexpglm}).
\end{lemma}

\begin{proof}
Let  \( |\lambda c_{i}(\gammav)| \le \lambdab_{1} \) for all \( i \).
Independence of the \( Y_{i} \)'s and (\ref{expzetai}) imply
in view of \( \sum_{i} c_{i}^{2}(\gammav) = 1 \) that
\begin{eqnarray*}
    &&
    \log \E \exp\biggl\{
        2 \lambda \frac{\gammav^{\T} \nabla \zeta}{(\gammav^{\T} V \gammav)^{1/2}}
        - 2 \lambda^{2} 
    \biggr\}
    =
    \sum_{i} \log \E \exp\Bigl\{
        2 \lambda c_{i}(\gammav)
        \frac{Y_{i} - \mY_{i}}{\expzeta_{i}}
    \Bigr\}
    - 2 \lambda^{2}
    \le
    0 .
\end{eqnarray*}
This implies \( (E\!D) \) with \( V(\thetav) \equiv V \).
\end{proof}

An important feature of the GLM is that the gradient
\( \nabla \zeta(\thetav) \) and hence, the corresponding matrix
\( V(\thetav) \) do not depend on \( \thetav \).
This automatically yields the condition \( (V) \) with \( \nu_{1} = 1 \) and
any \( \reps > 0 \).

%
Now we consider the rate function
\( \Lmgf(\thetav,\thetavs) \eqdef - \log \E \exp\{ \LL(\thetav,\thetavs) \} \).
It holds
\begin{eqnarray}
\label{LmgfGLM}
    \Lmgf(\thetav,\thetavs)
    &=&
    \mu \sum_{i} \bigl\{
         d(\Psi_{i}^{\T}\thetav) - d(\Psi_{i}^{\T}\thetavs)
        - \Psi_{i}^{\T}(\thetav - \thetavs) \mY_{i}
    \bigr\}
    \nn
    && \qquad - \,
    \sum_{i} \log \E \exp\bigl\{
        \mu \Psi_{i}^{\T}(\thetav - \thetavs) \bigl( Y_{i} - \mY_{i} \bigr)
    \bigr\}.
\end{eqnarray}
This function is smooth in \( \thetav \) and by (\ref{extremeexp})
\begin{eqnarray*}
    \nabla \Lmgf(\thetav,\thetavs) \big|_{\thetav = \thetavs}
    &=&
    \mu \sum_{i}\bigl\{
        \dot{d}(\Psi_{i}^{\T}\thetavs) - \mY_{i}
    \bigr\}  \Psi_{i}
    =
    0.
\end{eqnarray*}
Moreover,
\begin{eqnarray*}
    \nabla^{2} \Lmgf(\thetav,\thetavs)
    &=&
    \mu \sum_{i} \ddot{d}(\Psi_{i}^{\T}\thetav) \Psi_{i} \Psi_{i}^{\T}
    - \mu^{2} \sum_{i} s_{i}^{2}(\thetav) \Psi_{i} \Psi_{i}^{\T}
\end{eqnarray*}
where with \( \xi_{i} = Y_{i} - \mY_{i} \) and \( \uv = \thetav - \thetavs \)
\begin{eqnarray*}
    s_{i}(\thetav)
    \eqdef
    \left( \E \ex^{ \mu \Psi_{i}^{\T} \uv \xi_{i} } \right)^{-2}
    \left\{
        \E \xi_{i}^{2}\ex^{ \mu \Psi_{i}^{\T}\uv \xi_{i} }
        \E \ex^{ \mu \Psi_{i}^{\T}\uv \xi_{i} }
        -
        \left( \E \xi_{i} \ex^{ \mu \Psi_{i}^{\T}\uv \xi_{i} } \right)^{2}
    \right\}
\end{eqnarray*}
Particularly, \( s_{i}^{2}(\thetavs) = s_{i}^{2} = \Var Y_{i} \).
If \( \uv = \thetav - \thetavs \) is small then \( s_{i}(\thetav) \) is close to
\( s_{i} \).
The ``identifiability'' condition which would provide the concentration
property from Theorem~\ref{CLDboundMT} means that for some fixed positive constants
\( \mu \) and \( \fis \)
\begin{eqnarray}
\label{LmgfGLMu}
    \Lmgf(\thetav,\thetavs)
    \ge
    \fis^{2} (\thetav - \thetavs)^{\T} V (\thetav - \thetavs)
\end{eqnarray}
at least for all \( \thetav \) from a vicinity of \( \thetavs \).
The next lemma presents some simple sufficient conditions.
\begin{lemma}
\label{LexpGLM}
Let for a subset \( \Theta_{0} \subseteq \Theta \) hold:
\begin{eqnarray*}
    \frac{1}{2} 
    \sum_{i} \ddot{d}(\Psi_{i}^{\T}\thetav) \, \Psi_{i} \Psi_{i}^{\T}
    \ge
    2 \fis_{1} V_{1} \, ,
    \qquad
    \frac{1}{2} 
    \sum_{i} s_{i}^{2}(\thetav) \Psi_{i} \Psi_{i}^{\T}
    \le
    \fis^{2} V_{1} \, ,
    \qquad
    \thetav \in \Theta_{0} ,
\end{eqnarray*}
for some positive constants \( \fis_{1} = \fis_{1}(\Theta_{0}) \),
\( \fis = \fis(\Theta_{0}) \).
Then (\ref{LmgfGLMu}) is fulfilled with any \( \mu \) satisfying 
\( \mu \le \fis_{1}/\fis^{2} \).
\end{lemma}
\begin{proof}
With \( \mu \le \fis_{1} / \fis^{2} \) for any \( \thetav \in \Theta_{0} \)
\begin{eqnarray*}
    \frac{1}{2} 
    \nabla^{2} \Lmgf(\thetav,\thetavs)
    \ge
    2 \mu \fis_{1} V_{1} - \fis^{2} \mu^{2} V_{1}
    =
    (2 \mu^{-1} \fis_{1} - \fis^{2}) V
    \ge 
    \fis^{2} V  .
\end{eqnarray*}
Now the result follows by the second order Taylor expansion of
\( \Lmgf(\thetav,\thetavs) \) at \( \thetav = \thetavs \).
\end{proof}

Now we are ready to state the main result for the GLM estimation problem
which is a specification of Theorems~\ref{Tmainboundranking} and \ref{CLDboundMT2}.

\begin{theorem}
\label{TGLMexp}
Suppose that the \( Y_{i} \)'s are independent and
the point \( \thetavs \) is defined by (\ref{thetavsglm}).
Let there exist \( \lambdab_{1} > 0 \) and the values
\( \expzeta_{i} \) such that (\ref{expzetai}) is fulfilled.
Let, additionally, for some \( \mu > 0 \) and all \( \thetav \in \Theta \)
\begin{eqnarray*}
    \mu \expzeta_{i} |\Psi_{i}^{\T} (\thetav - \thetavs)| \le \lambdab_{1},
    \qquad
    i \le n  ,
\end{eqnarray*}
and with the matrices \( V, V_{1} \) from (\ref{Vexpglm}) and some \( \lambdab > 0 \):
\begin{eqnarray*}
    \sup_{\gammav \in S^{p}}
    \lambdab 
    \frac{\expzeta_{i} \bigl| \gammav^{\T} \Psi_{i} \bigr|}{(\gammav^{\T} V_{1} \gammav)^{1/2}} \le \lambdab_{1},
    \qquad
    i \le n  .
\end{eqnarray*}    
Fix any \( \mrho < 1 \) and \( \reps > 0 \) with 
\( \mrho \reps/(1- \mrho) \le \lambdab \). 
Define \( \penInt^{*} \) and \( \pen(\thetav) \) by (\ref{penIntdef}) and 
(\ref{penpenb}).
Then 
\begin{eqnarray*}
    &&
    \log \E \exp\Bigl\{
        \sup_{\thetav \in \Theta} \mrho \bigl[
                \LL(\thetav,\thetavs) + \Lmgf(\thetav,\thetavs)
                - \pen(\thetav)
        \bigr]
    \Bigr\}
    \nn
    &&
    \qquad \le 
    2 \reps^{2} \mrho^{2}/(1 - \mrho)
    + (1 - \mrho) \entr_{p}
    + \log (\penInt^{*}) .
\end{eqnarray*}    
Let also there exist \( \fis > 0 \)
such that the function \( \Lmgf(\thetav,\thetavs) \) from (\ref{LmgfGLM}) fulfills
\begin{eqnarray*}
    \Lmgf(\thetav,\thetavs)
    \ge
    \fis^{2} (\thetav - \thetavs)^{\T} V (\thetav - \thetavs) .
\end{eqnarray*}
Then for \( \srho = 1 - \fiss^{2}/\fis^{2} \)
\begin{eqnarray*}
    \P\bigl( 
        \Lmgf(\tilde{\thetav},\thetavs) > \zz 
    \bigr)
    & \le &
    \Crlq(\mrho,\srho) 
    \ex^{- \mrho \srho \zz } ,
    \nn
    \log \Crlq(\mrho,\srho)
    & \le &
    p C(\mrho)
    + p \log \bigl( |\fis^{2} (1 - \srho) (1 - \mrho)|^{-1/2} \bigr)
\end{eqnarray*}    
and for the confidence set 
\( \cc{E}(\zz) = \{ \thetav: \LL(\tilde{\thetav},\thetav) \le \zz \} \) holds
with \( \Crlq(\mrho) = \Crlq(\mrho,0) \)
\begin{eqnarray*}
    \P\bigl( \thetavs \not\in \cc{E}(\zz) \bigr)
    \le 
    \Crlq(\mrho) \ex^{- \mrho \zz  } .
\end{eqnarray*}
\end{theorem}

\section{Single-index regression}
\label{SexpSI}
In this section we illustrate the general results of Sections~\ref{Sriskboundex}
and \ref{Scorolexp} by the problem of estimating the index vector \( \thetav \)
in the so called single-index regression model.
Such models are frequently used in statistical modeling to overcome the
``curse of dimensionality'' problem, see \citeasnoun{stone1986}.

Let \( \Yn = (Y_{1},\ldots,Y_{n}) \) be an observed sample. We assume that
the \( Y_{i} \)'s are independent and the distribution of every \( Y_{i} \)
belongs to an exponential family \( \cc{P} \) with canonical parametrization:
\begin{eqnarray}
\label{modelnrd}
    Y_{i} \sim P_{\fu_{i}} \, ,
\end{eqnarray}
where the underlying parameter \( \fu_{i} \) can be different for each \( i \).
Regression analysis aims at explaining this parameter \( \fu_{i} \) as a function
of the explanatory vector \( X_{i} \in \R^{d} \):
\( \fu_{i} = f(X_{i}) \) for some \emph{regression function} \( f(\cdot) \).
We again consider a deterministic design \( X_{1},\ldots,X_{n} \).
The assumption \( \fu_{i} = f(X_{i}) \) reduces the original problem to
recovering the regression function \( f(\cdot) \) from the observed data.
However, in the case of a large \( d \) this problem is too complex because of
the design sparsity. This ``curse of dimensionality'' problem can be
avoided by some dimensionality reduction assumption.
Below we consider one possible assumption of this sort:
\begin{eqnarray}
\label{singleind}
    f(X_{i})
    =
    g(X_{i}^{\T} \thetav)
\end{eqnarray}
where \( g(\cdot) \) is a univariate \emph{link function}, while
\( \thetav \in \R^{d} \) is an \emph{index} vector.
This assumption effectively means that the explanatory vector \( X_{i} \) can
be projected on the index \( \thetav \) and this projection can be used
instead of the original vector without any information loss.
Therefore, the primary goal in estimation of a single index model is in recovering the
index vector \( \thetav \).
There is a number of results in the literature about the quality of estimation
of the index vector \( \thetav \).
We mention
\citeasnoun{Li1991},
\citeasnoun{ichi1993},
\citeasnoun{hallichi1993},
\citeasnoun{hris2001},
\citeasnoun{xia2002},
\citeasnoun{hart2002},
\citeasnoun{hris2003},
\citeasnoun{cook2008} among many others.

Below we assume that the link function \( g \) is given and it is sufficiently
smooth.
Note however, that the underlying model just follows
(\ref{modelnrd}). The considered problem is the problem of the best parametric
approximation of the function \( \fu(x) \) by  single index function
\( g(x^{\T} \thetav) \) with a fixed link function \( g(\cdot) \).
Such problem often occurs as an important building block in popular statistical
procedures like logit regression, projection pursuit of neuronal networks.
Note that for a linear link function \( g(\cdot) \) we come back to
generalized linear estimation.




The parametric assumption \( \fu(X_{i}) = g(X_{i}^{\T}\thetav) \)
leads to the log-likelihood
\( L(\thetav) = \sum_{i} \ell(Y_{i},g(X_{i}^{\T}\thetav)) \) where
\( \ell(y,\upsilon) = y \upsilon - d(\upsilon) + \ell(y) \) is the log-likelihood
function for \( \cc{P} \):
\begin{eqnarray}
\label{Lthetasi}
    L(\thetav) = \sum_{i} \ell(Y_{i},g(X_{i}^{\T}\thetav))
    =
    \sum_{i} \bigl\{
        Y_{i} g(X_{i}^{\T}\thetav) - d\bigl( g(X_{i}^{\T}\thetav) \bigr)
        + \ell(Y_{i})
    \bigr\} .
\end{eqnarray}
For some \( \mu > 0 \) whose value will be specified later, define
\begin{eqnarray*}
    \LL(\thetav)
    \eqdef
    \mu \sum_{i} \bigl\{
        Y_{i} g(X_{i}^{\T}\thetav) - d\bigl( g(X_{i}^{\T}\thetav) \bigr)
    \bigr\} .
\end{eqnarray*}    
We use the well known properties of the canonical exponential families:
\( E_{\upsilon} Y = \dot{d}(\upsilon)  \) which implies
\begin{eqnarray*}
    \E \LL(\thetav)
    &=&
    \mu \sum_{i} \bigl\{
        \dot{d}(\fu_{i}) g(X_{i}^{\T}\thetav) - d(g(X_{i}^{\T}\thetav))
    \bigr\},
    \\
    \nabla \E \LL(\thetav)
    &=&
    \mu \sum_{i} \bigl\{
        \dot{d}(\fu_{i}) - \dot{d}(g(X_{i}^{\T}\thetav))
    \bigr\} \dot{g}(X_{i}^{\T}\thetav) X_{i} \, .
\end{eqnarray*}
The target \( \thetavs \) maximizes \( \E \LL(\thetav) \):
\begin{eqnarray}
\label{thetavsmr}
    \thetavs
    =
    \argmax_{\thetav} \E \LL(\thetav)
    =
    \argmax_{\thetav}
    \sum_{i} \bigl\{
        \dot{d}(\fu_{i}) g(X_{i}^{\T}\thetav) - d(g(X_{i}^{\T}\thetav))
    \bigr\}.
\end{eqnarray}
This particularly yields
\begin{eqnarray*}
    \nabla \E \LL(\thetavs)
    =
    0.
\end{eqnarray*}
Next, for every \( \thetav \)
\begin{eqnarray*}
    \zeta(\thetav)
    & \eqdef &
    \LL(\thetav) - \E \LL(\thetav)
    =
    \mu \sum_{i}  \bigl\{  Y_{i} - \dot{d}(\fu_{i}) \bigr\} g(X_{i}^{\T}\thetav) ,
    \nn
    \nabla \zeta(\thetav)
    &=&
    \mu \sum_{i}  \bigl\{  Y_{i} - \dot{d}(\fu_{i}) \bigr\} \dot{g}(X_{i}^{\T}\thetav) X_{i},
\end{eqnarray*}
It is easy to see that condition \( (E) \) is fulfilled
if \( \fu_{i} + \mu \bigl\{ g(X_{i}^{\T}\thetav) - g(X_{i}^{\T}\thetavs) \bigr\} \in \cc{U} \)
for all \( i \) and all \( \thetav \in \Theta \).
Let \( \expzeta(\upsilon) \) be a function of \( \upsilon \)
which ensures for some fixed \( \lambdab_{1} > 0 \) that
\begin{eqnarray}
\label{expzetasi}
    \log E_{\upsilon} \exp\Bigl\{
        2 \lambda \frac{Y - \dot{d}(\upsilon)}{\expzeta(\upsilon)}
    \Bigr\}
    \le
    2 \lambda^{2},
    \qquad
    \lambda \le \lambdab_{1} \, .
\end{eqnarray}
Define
\begin{eqnarray}
\label{Vthetasi}
    V_{1}(\thetav)
    \eqdef
    \sum_{i} \expzeta^{2}(\fu_{i}) \,
        \bigl| \dot{g}(X_{i}^{\T}\thetav) \bigr|^{2} X_{i} X_{i}^{\T} ,
    \qquad 
    V(\thetav) 
    \eqdef 
    \mu^{2} V_{1}(\thetav).
\end{eqnarray}
Then for any \( \gammav \in S^{d} \)
\begin{eqnarray*}
    \E \exp\biggl\{ 
        2 \lambda 
        \frac{\gammav^{\T} \nabla \zeta(\thetav)}
             {(\gammav^{\T} V(\thetav) \gammav)^{1/2}}
        - 2 \lambda^{2} 
    \biggr\}
    \le
    1
\end{eqnarray*}
provided that \( \expzeta(\fu_{i}) |\lambda \gammav^{\T} X_{i}| \le \lambdab_{1} \)
for all \( i \), which implies \( (E\!D) \).

Now we consider the rate function \( \Lmgf(\thetav,\thetavs) \).
As
\( \log E_{\upsilon} \exp \bigl\{ \mu Y \bigr\} = d(\upsilon + \mu) - d(\upsilon)
\), it holds
\begin{eqnarray*}
    \Lmgf(\thetav,\thetavs)
    & \eqdef &
    - \log \E \exp\{ \LL(\thetav,\thetavs) \}
    \nn
    &=&
    \sum_{i}\Bigl\{
        d(\fu_{i})
        - d\bigl( \fu_{i} + \mu \delta_{i}(\thetav) \bigr)
    + \mu d(g(X_{i}^{\T}\thetav)) - \mu d(g(X_{i}^{\T}\thetavs))
    \Bigr\}
\end{eqnarray*}
with  \( \delta_{i}(\thetav) \eqdef g(X_{i}^{\T}\thetav) - g(X_{i}^{\T}\thetavs)
\). For the gradient \( \nabla \Lmgf(\thetav,\thetavs) \) holds:
\begin{eqnarray*}
    \nabla \Lmgf(\thetav,\thetavs)
    =
    \sum_{i}\Bigl\{
        \dot{d}\bigl( g(X_{i}^{\T}\thetav) \bigr)
        - \dot{d}\bigl( \fu_{i} + \mu \delta_{i}(\thetav) \bigr)
    \Bigr\} \mu \dot{g}(X_{i}^{\T}\thetav) X_{i} \, .
\end{eqnarray*}
The equation \( \nabla \E \LL(\thetavs) = 0 \) implies
\( \nabla \Lmgf(\thetav,\thetavs) \big|_{\thetav = \thetavs} = 0 \).
Moreover,
\begin{eqnarray*}
    &&
    \nabla^{2} \Lmgf(\thetav,\thetavs)
    = 
    \frac{\partial^{2} \Lmgf(\thetav,\thetavs)}{\partial \thetav^{2}}
    \nn
    && \quad =
    \sum_{i} \mu \Bigl\{
         \ddot{d}(g(X_{i}^{\T}\thetav)) 
        \bigl| \dot{g}(X_{i}^{\T}\thetav) \bigr|^{2}
    + \mu \bigl[ 
        \dot{d}(g(X_{i}^{\T}\thetav)) - \dot{d}(\fu_{i} - \mu \delta_{i}(\thetav)) 
    \bigl] \ddot{g}(X_{i}^{\T}\thetav)
    \Bigr\} X_{i} X_{i}^{\T} 
    \nn && \qquad - \, 
    \mu^{2} \sum_{i} \Bigl\{
        \ddot{d}(\fu_{i} + \mu \delta_{i}(\thetav)) 
        \bigl| \dot{g}(X_{i}^{\T}\thetav) \bigr|^{2}
    \Bigr\} X_{i} X_{i}^{\T} .
\end{eqnarray*}
The ``identifiability'' condition which would provide the concentration
property of \( \tilde{\thetav} \) in an elliptic neighborhood of the point
\( \thetavs \) means that
\begin{eqnarray*}
    n^{-1} \sum_{i}  \Bigl\{
        \ddot{d}(g(X_{i}^{\T}\thetav)) \bigl| \dot{g}(X_{i}^{\T}\thetav) \bigr|^{2}
        + \bigl[ \dot{d}(g(X_{i}^{\T}\thetav)) - \dot{d}(\fu_{i}) \bigr]
          \ddot{g}(X_{i}^{\T}\thetav)
    \Bigr\} X_{i} X_{i}^{\T}
    \ge
    D_{1}^{2}
\end{eqnarray*}
for a positive matrix \( D_{1}^{2} \). This condition ensures that with a proper
choice of \( \mu \), the value
\( \Lmgf(\thetav,\thetavs) \) satisfies
\( \Lmgf(\thetav,\thetavs) \ge C (\thetav - \thetavs)^{\T} D_{1}^{2} (\thetav - \thetavs) \)
for some \( C = C(D_{1}) > 0 \).

In the case of Gaussian regression \( Y_{i} \sim \cc{N}(\fu_{i},\sigma^{2}) \), it holds
\( d(\upsilon) = \upsilon^{2}/(2\sigma^{2}) \), so that
\( \dot{d}(\upsilon) = \upsilon/\sigma^{2} \) and
\( \ddot{d}(\upsilon) = \sigma^{-2} \). The identifiability condition reads now as
\begin{eqnarray*}
    (n \sigma^{2})^{-1} \sum_{i}  \Bigl\{
        \bigl| \dot{g}(X_{i}^{\T}\thetav) \bigr|^{2}
        + \bigl[ g(X_{i}^{\T}\thetav) - \fu_{i} \bigr] \ddot{g}(X_{i}^{\T}\thetav)
    \Bigr\} X_{i} X_{i}^{\T}
    \ge
    D_{1}^{2} .
\end{eqnarray*}
\begin{theorem}
\label{Tmainsi}
Suppose that \( Y_{i} \sim P_{\fu_{i}} \in \cc{P} \)  for some EFC \( \cc{P} \).
Let the point \( \thetavs \) be defined by (\ref{thetavsmr})
and \( \tilde{\thetav} = \argmax_{\thetav} \LL(\thetav) \) be its estimate.
Let also there exist \( \lambdab_{1} > 0 \) and the function
\( \expzeta(\upsilon) \) such that (\ref{expzetasi}) is fulfilled.
Let also for some \( \mu^{*} > 0 \)
\begin{eqnarray*}
    \fu_{i} + \mu^{*} \bigl\{ g(X_{i}^{\T}\thetav) - g(X_{i}^{\T}\thetavs) \bigr\} 
    \in \cc{U},
    \qquad 
    i \le n, \,\,
    \thetav \in \Theta,
\end{eqnarray*}
and for some \( \lambdab > 0 \) and the matrix \( V_{1}(\thetav) \) from (\ref{Vthetasi})
\begin{eqnarray*}
    \sup_{\gammav \in S^{p}}
    \frac{\expzeta(\fu_{i}) \bigl| \gammav^{\T} X_{i} \bigr|}
         {(\gammav^{\T} V_{1}(\thetav) \gammav)^{1/2}}
    \, \lambdab
    \le
    \lambdab_{1}     
    \qquad 
    i \le n, \,\,
    \thetav \in \Theta .
\end{eqnarray*}    
Then for any \( \mu \) with  \( 0 < \mu \le \mu^{*} \),
the conditions \( (E) \) and \( (E\!D) \)
are fulfilled with  \( V(\thetav) \) from (\ref{Vthetasi}).
For any \( \mrho < 1 \) and \( \reps > 0 \) with 
\( \mrho \reps/(1 - \mrho) \le \lambdab \), it holds
\begin{eqnarray*}
    &&
    \log \E \exp\Bigl\{
        \sup_{\thetav \in \Theta} \mrho \bigl[
                \LL(\thetav,\thetavs) + \Lmgf(\thetav,\thetavs)
                - \pen(\thetav)
        \bigr]
    \Bigr\}
    \nn
    &&
    \qquad \le 
    2 \reps^{2} \mrho^{2}/(1 - \mrho)
    + (1 - \mrho) \entr_{p}
    + \log (\penInt^{*}) ,
\end{eqnarray*}    
where \( \penInt^{*} \) and \( \pen(\thetav) \) are defined by (\ref{penIntdef}) and 
(\ref{penpenb}).

Let further there exist \( \fis > 0 \),
and a matrix \( V^{*} \) such that
\begin{eqnarray*}
    \Lmgf(\thetav,\thetavs)
    \ge
    \fis^{2}(\thetav - \thetavs)^{\T} D^{2} (\thetav - \thetavs),
    \qquad 
    V(\thetav) \le V^{*},
    \qquad 
    \thetav \in \Theta.
\end{eqnarray*}
Then for \( \srho = 1 - \fiss^{2}/\fis^{2} \), it holds
\begin{eqnarray*}
    \P\bigl( 
        \Lmgf(\tilde{\thetav},\thetavs) > \zz 
    \bigr)
    & \le &
    \Crlq(\mrho,\srho) 
    \ex^{- \mrho \srho \zz } ,
    \nn
    \log \Crlq(\mrho,\srho)
    & \le &
    p C(\mrho)
    + p \log \bigl( |\fis^{2} (1 - \srho) (1 - \mrho)|^{-1/2} \bigr)
\end{eqnarray*}    
and for the confidence set 
\( \cc{E}(\zz) = \{ \thetav: \LL(\tilde{\thetav},\thetav) \le \zz \} \) holds
with \( \Crlq(\mrho) = \Crlq(\mrho,0) \)
\begin{eqnarray*}
    \P\bigl( \thetavs \not\in \cc{E}(\zz) \bigr)
    \le 
    \Crlq(\mrho) \ex^{- \mrho \zz  } .
\end{eqnarray*}
\end{theorem}

\section{A penalized exponential bound for a random field}
\label{SriskboundG}
Let \( (\UU(\tv), \tv \in \Tv) \) be a random field on a probability space
\( (\Omega,\cc{F},\P) \), where \( \Tv \) is a separable
locally compact space.
For any \( \tv \in \Tv \) we assume the following exponential moment condition to
be fulfilled:

\begin{description}
  \item[\( \bb{(\cc{E})} \)]
 \textit{ For every \( \tv \in \Tv \)}
\begin{eqnarray*}
    \E \exp \{ \UU(\tv) \} = 1.
\end{eqnarray*}
\end{description}

The aim of this section is to establish a similar exponential bound for a
supremum of \( \UU(\tv) \) over \( \tv \in \Tv \).
A trivial corollary of the condition \( (\cc{E}) \) is that
if the set \( \Tv \) is finite with \( N = \# \Tv \),
then
\begin{eqnarray*}
    \E \exp\Bigl\{ \sup_{\tv \in \Tv} \UU(\tv) \Bigr\}
    \le
    N.
\end{eqnarray*}
Unfortunately, in the general case the supremum of \( \UU(\tv) \) over \( \tv \)
does not necessarily fulfill the condition of bounded exponential moments.
We therefore, consider a penalized version of the process \( \UU(\tv) \), that is,
we try to bound the exponential moment of \( \UU(\tv) - \pen(\tv) \) for some
\emph{penalty} function \( \pen(\tv) \). The goal is to find a possibly minimal
such function \( \pen(\tv) \) which provides
\begin{eqnarray*}
    \E \exp\Bigl\{ \sup_{\tv \in \Tv} \bigl[ \UU(\tv) - \pen(\tv) \bigr] \Bigr\}
    \le
    1.
\end{eqnarray*}
In the case of a finite set \( \Tv \), a natural candidate is
\( \pen(\tv) = \log(\#\Tv) \). Below we show how this simple choice can be
extended to the case of a general set \( \Tv \).
There exists a number of results about a supremum of a centered random field
which are heavily based on the theory of empirical processes.
See e.g. the monographes
\citeasnoun{VW1996},
\citeasnoun{sara2000},
\citeasnoun{massart2003},
and references therein.
Our approach is a bit different. First the process \( \UU(\tv) \) does not
need to be centered, instead we use the normalization
\( \E \exp \{ \UU(\tv) \} = 1 \).
Secondly we do not assume any particular structure of this process like
independence of observations, so the
methods of the empirical processes do not apply here.
Finally, our analysis is focuses on the penalty function \( \pen(\cdot) \)
rather then on the deviation probability of \( \max_{\tv} \UU(\tv) \).

\subsection{A local bound}
Define \( \ExpM(\tv) = \E \UU(\tv) \), \( \zeta(\tv) = \UU(\tv) - \E \UU(\tv)
\),
and denote \( \zeta(\tv,\tc) = \zeta(\tv) - \zeta(\tc) \) for \( \tv,\tc \in \Tv \).
We assume a nonnegative symmetric function \( \ExpL(\tv,\tc) \) is given such
that the following condition is fulfilled:
\begin{description}
\item[\( \bb{(\cc{E}\!\reps)} \)]\textit{
There exist numbers \( \reps > 0 \) and  \( \lambdab > 0 \), 
such that for any } \( \lambda \le \lambdab \)
\begin{eqnarray*}
    \sup_{\tv,\tc \in \Tv \, : \, \ExpL(\tv,\tc) \le \reps}
        \log \E \exp \Bigl\{
            2 \lambda \frac{\zeta(\tv,\tc)}{\ExpL(\tv,\tc)}
        \Bigr\}
    \le
    2 \lambda^{2} .
\end{eqnarray*}
\end{description}

Let \( \reps > 0 \) be shown in condition \( (\cc{E}\!\reps) \).
Define for any point \( \td \in \Tv \) the ``ball''
\begin{eqnarray*}
    \B(\reps,\td)
    =
    \bigl\{\tv: \ExpL(\tv,\td) \le \reps \bigr\}.
\end{eqnarray*}
To state the result, we have to introduce the notion of \emph{local entropy}.
We say that a discrete set \( \cc{D}(\reps,\cc{C}) \) is an \( \reps \)-net  in  \(
\cc{C} \subseteq \Tv \), if
\begin{eqnarray}
\label{localballs}
    \cc{C} \subset \bigcup_{\td \in \cc{D}(\reps,\cc{C})}  \B(\reps,\td).
\end{eqnarray}
By \( \Ns(\repsc,\reps,\td) \) for \( \repsc \le \reps \)
we denote the local covering number defined as the
minimal number of sets \( \B(\repsc,\cdot) \)  required to cover \(
\B(\reps,\td) \). With this covering number we associate the local entropy
\begin{eqnarray*}
    \entrl(\reps,\td)
    \eqdef
    \sum_{k=1}^{\infty} 2^{-k} \log \Ns(2^{-k}\reps,\reps,\td).
\end{eqnarray*}

Assume that \( \td \in \Tv \) is fixed. 
The following result controls the supremum in \( \tv \) of the penalized 
process \( \UU(\tv) - \pen(\tv) \) over the ball \( \B(\reps,\td) \).

\begin{theorem}
\label{Txibound}
Assume \( (\cc{E}) \) and \( (\cc{E}\!\reps) \).
For any \( \mrho \in (0,1) \) with
\( { \mrho \reps}/(1-\mrho) \le \lambdab \),
any \( \td \in \Tv \)
\begin{eqnarray*}
    \log \E \exp \Bigl\{
        \sup_{\tv \in \B(\reps,\td)}
        \mrho \bigl[ \UU(\tv) - \pen(\tv) \bigr]
    \Bigr\}
    \le
    \frac{2 \reps^{2} \mrho^{2}}{1 - \mrho} 
    + (1 - \mrho) \entrl(\reps,\td) - \mrho \pen_{\reps}(\td)   
\end{eqnarray*}
with
\begin{eqnarray*}
    \pen_{\reps}(\td)
    =
    \inf_{\tv \in \B(\reps,\td)} \pen(\tv).
\end{eqnarray*}
\end{theorem}

\begin{proof}
We begin with some result which bounds the stochastic component of the process
\( \UU(\tv) \) within the local ball \( \B(\reps,\td) \).

\begin{lemma}
\label{LstochUU}
Assume that \( \zeta(\tv) \) is a separable process satisfying
condition \( (\cc{E}\!\reps) \).
Then for any given \( \td \in \Tv \), \( \tdc \in \B(\reps,\td) \),
and \( \lambda \le \lambdab \)
\begin{eqnarray*}
    \log \E \exp\biggl\{
            \frac{\lambda}{\reps}
            \sup_{\tv \in \B(\reps,\td)} \zeta(\tv,\tdc)
        \biggr\}
    \le
    \entrl(\reps,\td) + 2 \lambda^{2} .
\end{eqnarray*}
\end{lemma}

\begin{proof}
The proof is based on  the standard chaining argument; see e.g.
\citeasnoun{VW1996}. Without loss of generality, we assume that
\( \entrl(\reps,\td) < \infty \). Then for any integer \( k \ge 0 \), there
exists a \( 2^{-k} \reps \)-net \( \cc{D}_{k}(\reps,\td) \) in the local ball \(
\B(\reps,\td) \) having the cardinality \( \Ns(2^{-k}\reps,\reps,\td) \).
%
%
 Using the nets  \( \cc{D}_{k}(\reps,\td)\) with \( k=1,\ldots, K-1\),  one can
 construct a chain connecting an arbitrary point \( \tv \)  in
\( \cc{D}_{K}(\reps,\td) \) and \( \tdc \).  It means
that  one can find   points 
\( \tv_{k} \in \cc{D}_{k}(\reps,\td), \) \( k=1,\ldots,K-1 \), such
that \( \ExpL(\tv_{k},\tv_{k-1}) \le 2^{-k+1}\reps \) for \(
k=1,\ldots,K \). Here \( \tv_{K} \) means \( \tv \) and
\( \tv_{0} \) means \( \tdc \).  Notice that
\(\tv_{k}\)  can be constructed recurrently:  \( \tv_{k-1} =
\tau_{k-1}(\tv_{k}), \ k=K,\ldots,1 \), where
\[
    \tau_{k-1}(\tv)
    =
    \argmin_{\tc \in \cc{D}_{k-1}(\reps,\td)} \ExpL(\tv,\tc).
\]
It obviously holds
\[
    \zeta(\tv,\tdc)
    =
    \sum_{k=1}^{K} \zeta( \tv_{k}, \tv_{k-1} ) .
\]
It holds for \( \xi( \tv_{k}, \tv_{k-1} ) =
\zeta(\tv_{k}, \tv_{k-1})/ \ExpL(\tv_{k}, \tv_{k-1}) \) that
\begin{eqnarray*}
    \zeta( \tv_{k}, \tv_{k-1} )
    =
    \ExpL( \tv_{k},\tv_{k-1} ) \xi( \tv_{k}, \tv_{k-1} )
    =
    2 \reps \,  c_{k} \, \xi( \tv_{k}, \tv_{k-1} )
\end{eqnarray*}
with \( c_{k} = \ExpL( \tv_{k},\tv_{k-1} )/(2 \reps) \le 2^{-k} \).
By condition \( (\cc{E}\!\reps) \)
\( \log \E \exp\bigl\{ 2 \lambda \xi( \tv_{k},\tv_{k-1}) \bigr\} \le
2 \lambda^{2} \).
Next,
\begin{eqnarray}
\label{e4x}
    \sup_{\tv \in \cc{D}_{k}(\reps,\td) } \zeta(\tv,\tdc)
    & \le &
    \sum_{k=1}^{K} \sup_{\tc \in \cc{D}_{k}(\reps,\td)}
    \zeta( \tc,\tau_{k-1}(\tc) )
    \nn
    & \le &
    2\reps \sum_{k=1}^{K}
    \sup_{\tc \in \cc{D}_{k}(\reps,\td)} c_{k}\xi( \tc,\tau_{k-1}(\tc) ) .
\end{eqnarray}
Since \( c_{k} \le 2^{-k} \),
the H\"older inequality and condition \( (\cc{E}\!\reps) \) imply 
\begin{eqnarray*}
    &&\log  \E \exp\biggl\{
            \frac{\lambda}{\reps} \sup_{\tv \in \cc{D}_{K}(\reps,\td)}
            \zeta(\tv,\tdc)
        \biggl\}
     \le
    \log \E \exp\biggl\{ 2 \lambda \sum_{k=1}^{K}  \sup_{\tc
    \in \cc{D}_{k}(\reps,\td)}
         c_{k} \xi( \tc,\tau_{k-1}(\tc) ) \biggr\}
    \nn
    &&\quad \le
    \sum_{k=1}^{K} 2^{-k} \log\biggl[
        \E \exp \Bigl\{
            \sup_{\tc \in \cc{D}_{k}(\reps,\td)}
            2^{k} c_{k}  \times 2\lambda \xi( \tc,\tau_{k-1}(\tc) )
        \Bigr\}
    \biggr]
    \nn
    &&\quad \le
    \sum_{k=1}^{K} 2^{-k} \log
    \biggl[
        \sum_{\tc \in \cc{D}_{k}(\reps,\td)}
        \E \exp\bigl\{
            2^{k} c_{k} \times 2 \lambda \xi( \tc,\tau_{k-1}(\tc) )
        \bigr\}
    \biggr]
    \nn
    &&\quad  \le
    \sum_{k=1}^{K} 2^{-k}
    \bigl\{ \log \Ns(2^{-k}\reps,\reps,\td) + 2 \lambda^{2} \bigr\} .
\end{eqnarray*}
These inequalities and the separability of \( \zeta(\tv,\tdc) \) yield
\begin{eqnarray*}
    \log \E \exp\biggl\{
        \frac{\lambda}{\reps} \sup_{\tv \in \B(\reps,\td)}
        \zeta(\tv,\tdc)
    \biggl\}
    =
    \lim_{K\rightarrow \infty}\log \E \exp\biggl\{
        \frac{\lambda}{\reps} \sup_{\tv \in \cc{D}_{K}(\reps,\td)}
        \zeta(\tv,\tdc)
    \biggl\}
    \\
    \le
    \sum_{k=1}^{\infty}
        2^{-k} \bigl\{  2 \lambda^{2} + \log \Ns(2^{-k}\reps,\reps,\td) \bigr\}
    \le
    2 \lambda^{2} + \entrl(\reps,\td)
\end{eqnarray*}
which completes the proof of the lemma.
\end{proof}

Now define for a fixed a point \( \td \) 
\begin{eqnarray*}
    \tdc
    =
    \argmin_{\tv \in \B(\reps,\td)} \{ \ExpM(\tv) + \pen(\tv) \},
\end{eqnarray*}
where \( \ExpM(\tv) = - \E \UU(\tv) \).
If there are many such points, then take any of them as \( \tdc \).
Obviously
\begin{eqnarray*}
    \sup_{\tv \in \B(\reps,\td)}
        \bigl\{ \UU(\tv) - \pen(\tv) \bigr\}
    \le
    \UU(\tdc) - \pen(\tdc)
    + \sup_{\tv \in \B(\reps,\td)} \zeta(\tv,\tdc).
\end{eqnarray*}
Therefore, by the H\"older inequality and Lemma~\ref{LstochUU} with
\( \lambda = \reps \mrho /(1 - \mrho) \)
\begin{eqnarray*}
\label{BcsupUU}
    &&
    \log \E \exp \Bigl\{
        \sup_{\tv \in \B(\reps,\td)}
        \mrho \bigl[ \UU(\tv) - \pen(\tv) \bigr]
    \Bigr\}
    \\
    &&
    \quad \le
    \mrho \log \E \exp \bigl\{
        \UU(\tdc) - \pen(\tdc)
    \bigr\}
    + \, (1 - \mrho)
    \log \E \exp \Bigl\{ \frac{\mrho}{1 - \mrho}
         \sup_{\tv \in \B(\reps,\td)}
             \zeta(\tv,\tdc)
    \Bigr\}
    \\
    &&
    \quad \le
    2 \reps^{2} \mrho^{2}/ (1 - \mrho) 
    + (1 - \mrho) \entrl(\reps,\td) - \mrho \pen(\tdc)
    \\
    &&
    \quad \le
    {2 \reps^{2} \mrho^{2}}/(1 - \mrho)
    + (1 - \mrho) \entrl(\reps,\td) - \mrho \pen_{\reps}(\td)     .
\end{eqnarray*}
which is the assertion of the theorem.
\end{proof}

\subsection{A global exponential bound for the penalized process}
This section presents some sufficient conditions on the penalty function
\( \pen(\tv) \) which ensure the general exponential bound for the penalized process
\( \UU(\tv) - \pen(\tv) \).
For simplicity we assume that the local entropy numbers
\( \entrl(\reps,\tv) \) are uniformly bounded by a constant \( \entrl^{*}(\Tv) \).
Let also \( \mes \) be a \( \sigma \)-finite measure on the space \( \Tv \) and
 \( \mes(A) \) stand for the \( \mes \)-measure of a set \( A \subset \Tv \).
The standard proposal for \( \mes \) is the usual Lebesgue measure.

\begin{theorem}
\label{Tmainboundexp2}
Assume \( (\cc{E}) \) and \( (\cc{E}\!\reps) \) with some fixed \( \reps \) and \( \lambdab \).
Let \( \mrho < 1 \) be such that
\( { \mrho \reps}/{(1-\mrho) } \le \lambdab \).
Let also \( \entrl(\reps,\tv) \le \entrl^{*}(\Tv) \) for all \( \tv \in \Tv 
\). Let a \( \sigma \)-finite measure \( \mes \) on \( \Tv \) be such that 
for some \( \nu \ge 1 \)
\begin{eqnarray}
\label{Btvtd}
    \sup_{\tv,\tc: \, \ExpL(\tv,\tc) \le \reps}
        \frac{\mes(\B(\reps,\tv))}{\mes(\B(\reps,\tc))}
    \le
    \nu .
\end{eqnarray}
Finally, let a function \( \pen(\tv) \) satisfy
\begin{eqnarray*}
    \penH_{\reps}(\mrho)
    \eqdef
    \log \int_{\Tv} \frac{1}{\mes(\B(\reps,\td))}
    \exp\bigl\{ - \mrho \pen_{\reps}(\td) \bigr\}  d \mes(\td)
    < \infty
\end{eqnarray*}
with \( \pen_{\reps}(\td)
    =
    \inf_{\tv \in \B(\reps,\td)} \pen(\tv) \).
Then
\begin{eqnarray}
\label{UUbound}
    \E \exp\Bigl\{
        \sup_{\tv \in \Tv} \mrho \bigl[ \UU(\tv) - \pen(\tv) \bigr]
    \Bigr\}
    \le 
    \Crlq(\mrho,\reps) ,
\end{eqnarray}
where
\begin{eqnarray}
\label{riskre}
    \log \Crlq(\mrho,\reps)
    =
    \frac{2 \reps^{2} \mrho^{2}}{1 - \mrho} 
    + (1 - \mrho) \entrl^{*}(\Tv) + \log \nu
    + \penH_{\reps}(\mrho) .
\end{eqnarray}    
\end{theorem}

\begin{proof}
We begin with a simple technical result which bounds the maximum of a given function
via the weighted integral of the local maxima.
\begin{lemma}
\label{Lmaxloc}
Let \( f(\tv) \) be a nonnegative function on \( \Tv \subset \R^{p} \) and let for every
point \( \tv \in \Tv \) a vicinity \( A(\tv) \) be fixed such that
\( \tc \in A(\tv) \) implies \( \tv \in A(\tc) \). Let also
the measure \( \mes\bigl( A(\tv) \bigr) \) of the set \( A(\tv) \)
fulfill for every \( \td \in \Tv \)
\begin{eqnarray}
\label{Atvtd}
    \sup_{\tv \in A(\td)} \frac{\mes\bigl( A(\tv) \bigr)}{\mes\bigl( A(\td) \bigr)}
    \le
    \nu .
\end{eqnarray}
Then
\begin{eqnarray*}
    \sup_{\tv \in \Tv} f(\tv)
    \le
    \nu \int_{\Tv} f^{*}(\tv) \frac{1}{\mes\bigl( A(\tv) \bigr)} d\mes(\tv)
\end{eqnarray*}
with 
\begin{eqnarray*}
    f^{*}(\tv)
    \eqdef
    \sup_{\tc \in A(\tv)} f(\tc).
\end{eqnarray*}
\end{lemma}

\begin{proof}
For every \( \td \in \Tv \)
\begin{eqnarray*}
    \int_{\Tv} f^{*}(\tv) \frac{1}{\mes\bigl( A(\tv) \bigr)} d\mes(\tv)
    & \ge &
    \int_{A(\td)} f^{*}(\tv) \frac{1}{\mes\bigl( A(\tv) \bigr)} d\mes(\tv)
    \nn
    & \ge &
    f(\td)\int_{A(\td)} \frac{1}{\mes\bigl( A(\tv) \bigr)} d\mes(\tv)
\end{eqnarray*}
because \( \tv \in A(\td) \) implies \( \td \in A(\tv) \) and hence,
\( f(\td) \le f^{*}(\tv) \). Now by (\ref{Atvtd})
\begin{eqnarray*}
    \int_{\Tv} f^{*}(\tv) \frac{1}{\mes\bigl( A(\tv) \bigr)} d\mes(\tv)
    \ge
    \frac{f(\td)}{\nu} \int_{A(\td)} \frac{1}{\mes\bigl( A(\td) \bigr)} d\mes(\tv)
    =
    f(\td) /\nu
\end{eqnarray*}
as required.
\end{proof}

This result applied to
\( f(\tv) = \exp\bigl\{ \mrho \bigl[ \UU(\tv) - \pen(\tv) \bigr] \bigr\} \)
and \( A(\tv) = \B(\reps,\tv) \) implies
\begin{eqnarray*}
    \sup_{\tv \in \Tv} \exp\Bigl\{
        \mrho \bigl[ \UU(\tv) - \pen(\tv) \bigr]
    \Bigr\}
    \le
    \nu \int_{\Tv}
        \sup_{\tv \in \B(\reps,\td)}
        \exp\Bigl\{ \mrho \bigl[ \UU(\tv) - \pen(\tv) \bigr] \Bigr\}
        \frac{d\mes(\td)}{\mes\bigl( \B(\reps,\td) \bigr)} .
\end{eqnarray*}
This implies by Theorem~\ref{Txibound}
\begin{eqnarray*}
    &&
    \log \E \sup_{\tv \in \Tv} \exp\Bigl\{
        \mrho \bigl[ \UU(\tv) - \pen(\tv) \bigr]
    \Bigr\}
    \nn
    && \qquad \le
    \frac{2 \reps^{2} \mrho^{2}}{1 - \mrho} + (1 - \mrho) \entrl^{*}(\Tv)
    +
    \log
    \biggl\{ \nu \int_{\Tv} \exp\bigl\{ - \mrho \pen_{\reps}(\td) \bigr\}
    \frac{d\mes(\td)}{\mes\bigl( \B(\reps,\td) \bigr)}  \biggr\}
    \nn
    && \qquad \le
    \frac{2 \reps^{2} \mrho^{2}}{1 - \mrho} + (1 - \mrho) \entrl^{*}(\Tv)
    + \log (\nu)
    + \penH_{\reps}(\mrho) 
\end{eqnarray*}
and the assertion follows.
\end{proof}

\subsection{Smooth case}
Here we discuss the special case when \( \Tv \subset \R^{p} \),
the process \( \UU(\tv) \) and its stochastic component \( \zeta(\tv) \) are
absolutely continuous and the gradient
\( \nabla \zeta(\tv) \eqdef d \zeta(\tv)/d\tv \) has bounded exponential moments.
We also assume that \( \mes \) is the Lebesgue measure on \( \Tv \).
Suppose the following condition is fulfilled:
\begin{description}
\item[\( \bb{(\cc{E}\! D)} \)]\textit{
There exist \( \lambdab > 0 \) and for each \( \tv \in \Tv \),
a symmetric non-negative matrix \( \VV(\tv) \)
such that for any } \( \lambda \le \lambdab \)
\begin{eqnarray*}
    \sup_{\tv \in \Tv} \sup_{\gammav \in S^{p}}
    \log \E \exp \Bigl\{
       2 \lambda \frac{\gammav^{\T}\nabla \zeta(\tv)}{\| \VV(\tv) \gamma \|}
    \Bigr\} \le 2 \lambda^{2} .
\end{eqnarray*}
\end{description}

The matrix function \( \VV(\tv) \) can be used for defining a natural topology in
\( \Tv \). Namely,
for any  \( \tv,\tc \in \Tv \) define
\( \normc = \| \tv - \tc \| \), \( \gammav = (\tv - \tc)/\normc \) and
\begin{eqnarray*}
    \ExpL^{2}(\tv,\tc)
    \eqdef
    \| \tv - \tc \|^{2} \int_{0}^{1} \gammav^{\T} \VV^{2}(\tv + t \normc \gammav) \gammav \, dt  .
\end{eqnarray*}
Next, introduce for each \( \td \in \Tv \) and \( \reps > 0 \) the set
\begin{eqnarray*}
    \B(\reps,\td)
    \eqdef
    \{ \tv: \ExpL(\tv,\td) \le \reps \}
\end{eqnarray*}

To state the result, we need one more condition on the uniform continuity of the matrix 
\( H(\tv) \) in \( \tv \).
\begin{description}
\item[\( \bb{(H)} \)]\textit{
There exist constants
\( \reps > 0 \) and \( \nu_{1} \ge 1 \) such that
\begin{eqnarray*}
    \qquad
    \sup_{\tv,\tc: \ExpL(\tv,\tc) \le \reps} \,\,
    \sup_{\gammav \in S^{p}}
    \frac{\gammav^{\T} \VV^{2}(\tv) \gammav}{\gammav^{\T} \VV^{2}(\tc) \gammav}
    \le \nu_{1} \, .
\end{eqnarray*}
}
\end{description}

\begin{theorem}
\label{Tsmoothpen}
Let \( (\cc{E}) \) be satisfied.
Suppose that \( (\cc{E}\! D) \) holds with some \( \lambdab \) and a matrix function
\( \VV(\tv) \) which fulfills \( (H) \).
If for some \( \mrho \in (0,1) \) and \( \reps > 0 \) with
\( { \mrho \reps}/(1-\mrho) \le \lambdab \),  the penalty function
\( \pen(\tv) \) fulfills
\begin{eqnarray*}
    \penH_{\reps}(\mrho)
    \eqdef
    \log \biggl\{ 
        \omega_{p}^{-1} \reps^{-p}
        \int_{\Tv} \det(\VV(\td))
        \exp\bigl\{ - \mrho \pen_{\reps}(\td) \bigr\}  d \td
    \biggr\}
    < \infty
\end{eqnarray*}
with \( \pen_{\reps}(\td)
    =
    \inf_{\tv \in \B(\reps,\td)} \pen(\tv) \),
then
\begin{eqnarray}
\label{boundsmooth}
    \E \exp\Bigl\{
        \sup_{\tv \in \Tv} \mrho \bigl[ \UU(\tv) - \pen(\tv) \bigr]
    \Bigr\} 
    \le 
    \Crlq(\mrho,\reps)
\end{eqnarray}
where
\begin{eqnarray*}
    \log \Crlq(\mrho,\reps)
    &=& 
    \frac{2 \reps^{2} \mrho^{2}}{1 - \mrho} 
    + (1 - \mrho) \entr_{p}
    + \penH_{\reps}(\mrho)
    + p \log (\nu_{1}) 
\end{eqnarray*}    
with \( \entr_{p} \) being the usual entropy number for the Euclidean ball
in \( \R^{p} \).
\end{theorem}

\begin{proof}
First we show that the differentiability condition \( (\cc{E}\! D) \) implies the local
moment condition \( (\cc{E}\!\reps) \).

\begin{lemma}
\label{Lsmu}
Assume that \( (\cc{E}\! D) \) holds with some \( \lambdab \). Then
for any \( \td \in \Tv \) and any \( \lambda \) with
\( |\lambda| \le \lambdab / \nu_{1}^{1/2} \), it holds
\begin{eqnarray}
\label{zsmu}
    \sup_{\tv \in \B(\reps,\td)} \log \E \exp \biggl\{
        2 \lambda \frac{\zeta(\tv,\td)}{\ExpL(\tv,\td)}
    \biggr\}
    \le
    2 \lambda^{2}.
\end{eqnarray}
\end{lemma}

\begin{proof}
For   \( \tv \in \B(\reps,\td) \), denote \( \normc = \| \tv - \td \| \),
\( \gammav = (\tv - \td)/\normc \).
With this notation
\begin{eqnarray*}
    \zeta(\tv,\td)
    =
    \normc \gammav^{\T}\int_{0}^{1} \nabla \zeta(\td + t \normc \gammav) dt .
\end{eqnarray*}
The condition \( (H) \) implies for every \( t \in [0,1] \) that
\begin{eqnarray*}
    \lambda \frac{u \| \VV(\td + t \normc \gammav) \gammav \|}{\ExpL(\tv,\td)}
    \le
    \lambda \nu_{1}^{1/2}
    \le
    \lambdab .
\end{eqnarray*}
Now the H\"older inequality and \( (\cc{E}\! D) \) yield
\begin{eqnarray*}
    &&
    \log \E \exp \biggl\{
        2 \lambda \frac{\zeta(\tv,\td)}{\ExpL(\tv,\td)} - \lambda^{2}
    \biggr\}
    \nn
    && \quad =
    \log \E \exp \biggl\{
        \int_{0}^{1}
        \gammav^{\T} \biggl[
            \frac{2 \lambda \normc}{\ExpL(\tv,\td)}
            \nabla \zeta(\td + t \normc \gammav)
            - \frac{2 \lambda^{2} \normc^{2}}{\ExpL^{2}(\tv,\td)}
            \VV^{2}(\td + t \normc \gammav) \gammav
        \biggr]  dt
    \biggr\}
    \nn
    && \quad  \le
    \int_{0}^{1} \log \E \exp \biggl\{
        \gammav^{\T} \biggl[
            \frac{2 \lambda \normc}{\ExpL(\tv,\td)}
            \nabla \zeta(\td + t \normc \gammav)
            - \frac{2\lambda^{2} \normc^{2}}{\ExpL^{2}(\tv,\td)}
            \VV^{2}(\td + t \normc \gammav) \gammav
        \biggr]
    \biggr\}\, dt
    \nn
    && \quad \le
    0
\end{eqnarray*}
as required.
\end{proof}

Next we show that condition \( (H) \) implies (\ref{Btvtd}).
Consider for every \( \td \in \Tv \) an elliptic neighborhood
\( \BB(\reps,\td) = \{ \tv: \| \VV(\td) (\tv - \td) \| \le \reps \} \).

\begin{lemma}
\label{LBtvtd}
Assume \( (H) \). Then
\begin{enumerate}
  \item for any \( \reps > 0 \)  and any \( \tv \in \Tv \)
\begin{eqnarray}
\label{BBBnu}
\begin{array}{rcccl}
    \BB(\nu_{1}^{-1/2} \reps,\tv)
    & \subset &
    \B(\reps,\tv)
    & \subset &
    \BB(\nu_{1}^{1/2} \reps,\tv),
    \\
    \B(\nu_{1}^{-1/2} \reps,\tv)
    & \subset &
    \BB(\reps,\tv)
    & \subset &
    \B(\nu_{1}^{1/2} \reps,\tv).
\end{array}
\end{eqnarray}
  \item For every \( \tv \in \Tv \),
\begin{eqnarray}
    \nu_{1}^{-p/2}
    \le
    \reps^{-p} \mes(\B(\reps,\tv))}{{\det(\VV(\tv))} / \omega_{p}
    \le
    \nu_{1}^{p/2} ,
\end{eqnarray}
where \( \omega_{p} \) is the Lebesgue measure of the unit ball in \( \R^{p} \).
    \item condition (\ref{Btvtd}) holds with \( \nu = \nu_{1}^{p} \).
\end{enumerate}
\end{lemma}
\begin{proof}
Condition \( (H) \) implies that for any \( \td \in \Tv \)
and \( \tv \in \B(\reps,\td) \) that
\begin{eqnarray*}
    \nu_{1}^{-1} \gammav^{\T} \VV^{2}(\td) \gammav
    \le
    \int_{0}^{1} \gammav^{\T} \VV^{2}(\td + t \normc \gammav) \gammav \, dt
    \le
    \nu_{1} \gammav^{\T} \VV^{2}(\td) \gammav
\end{eqnarray*}
with \( \normc = \| \tv-\td \| \) and \( \gammav = (\tv-\td)/\normc \),
which yields the first assertion of the lemma.

The Lebesgue measure of the ellipsoid
\( \BB(\reps,\tv) \) is equal to \( \omega_{p} \reps^{p} \big/ {\det(\VV(\tv))} \).
This and (\ref{BBBnu}) imply the second assertion.
This, in turns, implies (\ref{Btvtd}) in view of \( (H) \).
\end{proof}

The next result claims that in the smooth case the local entropy number
\( \entrl(\reps,\td) \) is similar to the usual Euclidean situation.

\begin{lemma}
\label{Lentrls}
Assume \( (H) \). Then
\(
    \sup_{\tv\in \Theta} \entrl(\reps,\tv) \le \entr_{p} + p  \log (\nu_{1})
    \).
\end{lemma}

\begin{proof}
Fix any \( \td \in \Tv \).
Linear transformation with the matrix \( \VV^{-1}(\td) \) reduces the situation
to the case when \( \VV(\td) \equiv I \) and \( \BB(\repsc,\td) \)
is a usual Euclidean ball for any \( \repsc \le \reps \).
Moreover, by \( (H) \), each elliptic set \( \BB(\repsc,\tv) \) for
\( \tv \in \B(\reps,\td) \) is nearly an Euclidean ball
in the sense that the ratio of
its largest and smallest axes (which is the ratio of
the largest and smallest eigenvalues of
\( \VV^{-1}(\td) \VV^{2}(\tv) \VV^{-1}(\td) \)) is bounded by \( \nu_{1} \).
Therefore, for any \( \repsc \le \reps \), a Euclidean net
\( \cc{D}^{e}(\repsc/\nu_{1}) \) with the step
\( \repsc / \nu_{1} \)
ensures a covering of \( \B(\reps,\td) \) by the sets
\( \B(\repsc,\td) \), \( \td \in \cc{D}^{e}(\repsc/\nu_{1}) \).
Therefore, the corresponding covering number is
bounded by \( (\nu_{1}\reps/\repsc)^{p} \)
yielding the claimed bound for the local entropy.
\end{proof}

Now the result of theorem~\ref{Tsmoothpen} is reduced to the statement of
Theorem~\ref{Tmainboundexp2}.
\end{proof}

Computing of the penalty simplifies a lot when the matrix \( \VV(\tv) \) is
uniformly bounded by a matrix \( \VV^{*} \), or, equivalently, 
condition \( (H) \) is fulfilled for \( \VV(\tv) \equiv \VV^{*} \). 
Then one can define \( \pen(\tv) \) as a function of the norm
\( \| \VV^{*} (\tv - \tvs) \| \) for a fixed 
\( \tvs \).

\begin{theorem}
\label{Tqrating}
Assume additionally to the conditions of Theorem~\ref{Tsmoothpen} that
\( \VV(\tv) \le \VV^{*} \) for a symmetric matrix \( \VV^{*} \). 
Suppose that \( \penb(t) \) is a monotonously decreasing positive 
function on \( [0,+\infty ) \) satisfying 
\begin{eqnarray}
\label{penInt}
    \penInt^{*}
    \eqdef
    \omega_{p}^{-1} \int_{\R^{p}} \penb(\| \uv \|) d\uv 
    = 
    p \int_{0}^{\infty}  \penb(t) t^{p-1} dt
    < \infty .
\end{eqnarray}
Define 
\begin{eqnarray*}
    \pen(\tv) 
    = 
    - \mrho^{-1} \log \penb\bigl( \reps^{-1} \| \VV^{*} (\tv - \tvs) \| + 1 \bigr)
\end{eqnarray*}    
Then
%
\begin{eqnarray}
\label{boundsmoothp}
    \E \exp\Bigl\{
        \sup_{\tv \in \Tv} \mrho \bigl[ \UU(\tv) - \pen(\tv) \bigr]
    \Bigr\} 
    \le 
    \Crlq(\mrho,\reps)
\end{eqnarray}
with
\begin{eqnarray*}
    \log \Crlq(\mrho,\reps)
    &=& 
    \frac{2 \reps^{2} \mrho^{2}}{1 - \mrho} 
    + (1 - \mrho) \entr_{p}
    + \log (\penInt^{*}),
\end{eqnarray*}    
where \( \omega_{p} \) is the volume of the unit ball in \( \R^{p} \).

\end{theorem}

\begin{proof}
Let us fix \( \td \in \Tv \).
Definition of the semi-metric \( \ExpL \) and condition \( (H) \) imply for every 
\( \tv \in \B(\reps,\td) \) that
\begin{eqnarray*}
    \| \VV^{*}(\td - \tv) \|
    \le 
    \reps .
\end{eqnarray*}    
The triangle inequality and \( (H) \) now imply  for this \( \tv \) that
\begin{eqnarray*}
    \reps^{-1} \| \VV^{*}(\tv - \tvs) \| + 1
    \ge 
    \reps^{-1} \| \VV^{*}(\td - \tvs) \| 
\end{eqnarray*}    
and \( \pen_{\reps}(\td) \ge - \mrho^{-1} \log \penb\bigl( \reps^{-1} \| \VV^{*} (\td - \tvs) \| \bigr) \).
Therefore, it follows by change of variables \( \uv = \reps \VV^{*}(\tv - \tvs) \) that
\begin{eqnarray*}
    \omega_{p}^{-1} \reps^{-p} \int_{\Tv} \det(\VV^{*})
        \exp\bigl\{ - \mrho \pen_{\reps}(\tv) \bigr\}  d \tv
    & \le &
    \omega_{p}^{-1} \int_{\R^{p}} 
        \penb(\| \uv \| ) d\uv 
    \nn 
    & \le &
    p \int_{0}^{\infty} \penb(t) t^{p-1} dt
    = 
    \penInt^{*},
\end{eqnarray*}
and the result follows from Theorem~\ref{Tsmoothpen}.
\end{proof}

Natural candidates for the function \( \penb(\cdot) \) 
and the corresponding \( \penInt^{*} \)-values are:
\begin{eqnarray*}
    \begin{array}{rclrcl}
    \penb_{1}(t)
    &=&
    e^{ - \delta_{1} (t-1)_{+}^{2}},
    &
    \penInt^{*}_{1}
    &=&
    1 + 
    \omega_{p}^{-1} (\pi /\delta_{1})^{p/2},    
    \\
    \penb_{2}(t) 
    &=& 
    \| 1 + t \|^{- p - \delta_{2}},
    &
    \penInt^{*}_{2}
    &=&
    p /\delta_{2} \, ,
\end{array}  
\end{eqnarray*}
where \( \delta_{1},\delta_{2} > 0 \) are some constants.
The result of Theorem~\ref{Tqrating} yields
\begin{corollary}
\label{Cqrating}
Under conditions of Theorem~\ref{Tqrating}, the bound (\ref{boundsmoothp}) 
holds with 
\begin{eqnarray*}
    \pen_{1}(\tv) 
    &=& 
    \mrho^{-1} \delta_{1} \, \reps^{-2} \| \VV^{*} (\tv - \tvs) \|^{2} ,
    \nn
    \log \Crlq_{2}(\mrho,\reps)
    &=& 
    \frac{2 \reps^{2} \mrho^{2}}{1 - \mrho} 
    + (1 - \mrho) \entr_{p}
    + \log (1 + \omega_{p}^{-1} |\pi/\delta_{1}|^{p/2}) .
    \nn
    \pen_{1}(\tv) 
    &=& 
    - \mrho^{-1} (p + \delta_{2}) 
    \log \bigl( \reps^{-1} \| \VV^{*} (\tv - \tvs) \| + 2 \bigr),
    \nn
    \log \Crlq_{1}(\mrho,\reps)
    &=& 
    \frac{2 \reps^{2} \mrho^{2}}{1 - \mrho} 
    + (1 - \mrho) \entr_{p}
    + \log (p / \delta_{2}) ,
\end{eqnarray*}    
\end{corollary}
Sometimes it is useful to combine the functions \( \penb_{1}(\cdot) \) and
\( \penb_{2}(\cdot) \) in the form
\begin{eqnarray}
\label{penbtr}
    \penb(t) = \penb_{1}(t) \bb{1}(t \ge r) + \penb_{2}(t) \bb{1}(t \le r)
\end{eqnarray}    
for a properly selected \( r \) which still ensures (\ref{penInt}) with 
\begin{eqnarray*}
    \penInt^{*} 
    \le
    \omega_{p}^{-1} |\pi/\delta_{1}|^{p/2} 
    +
    p r^{-\delta_{2}} /\delta_{2}  .
\end{eqnarray*}

\bibliography{exp_ts,listpubm-with-url}
\end{document}

\section*{Conclusion and outlook}
The paper presents an approach for studying the non-asymptotic behavior of a
general quasi MLE. The established exponential bounds yield strict
concentration properties as well as  polynomial
risk bounds for this estimate under very mild and general conditions.
The results can be applied for many
specific statistical problems like building some confidence sets of testing.

It is worth noting that the bound are quite rough because of the basic
tools: an exponential Markow inequality. However, the concentration
property of the estimate \( \tilde{\thetav} \) combined with the standard
local approximation technique can be used to obtain the classical risk bounds.

An important issue is the optimality of the suggested penalty function
\( \pen(\thetav) \). In some problems like multiscale estimation and testing
this function enters in the upper bound for the risk and its careful evaluation
remains a challenging question.

An extension of the proposed approach to the problem of semiparametric
estimation is an important issue as well. As a specific example, the estimation
of the index vector for an unknown link function can be mentioned.